\documentclass[a4paper,12pt]{article}

\usepackage{cmap}
\usepackage[T2A]{fontenc}
\usepackage[utf8]{inputenc} % [cp866][utf8]
\usepackage[english]{babel}
\usepackage{amssymb,amsmath,amsfonts}
\usepackage{ifthen}
\usepackage{tikz}

\makeatletter
\def\@seccntformat#1{\csname the#1\endcsname.\ } % точка после номера раздела
\def\@biblabel#1{#1.} % формат номеров в списке литературы
\makeatother

\date{}

\newenvironment{proof}[1][\hspace{-1.0ex}]%
{\par\addvspace{1mm}{\sc Proof\hspace{1.0ex}{#1}.} }%
{$\blacktriangle$\par\addvspace{1mm}}

\newif\ifNoRemark
\def\addtheorem#1#2#3#4{
\ifthenelse{\equal{#2}{}}{}%
{\ifthenelse{\expandafter\isundefined\csname the#2\endcsname}{\newcounter{#2}}{}}
\newenvironment{#1}[1][\global\NoRemarktrue]% No Remark by default
{\par\addvspace{2mm plus 0.5mm minus 0.2mm}\noindent % new paragraph without indent
{\bf #3}\ifthenelse{\equal{#2}{}}{}%
{\refstepcounter{#2}{\bf ~\csname the#2\endcsname}}%
{\bf \vphantom{##1}\ifNoRemark.\ \else\ (##1).\fi}\begingroup #4}%
{\endgroup\par\addvspace{1mm plus 0.5mm minus 0.2mm}\global\NoRemarkfalse}
\expandafter\newcommand\csname b#1\endcsname{\begin{#1}}
\expandafter\newcommand\csname e#1\endcsname{\end{#1}}
}

\addtheorem{theorem}{thrm}{Theorem}{\sl}
\addtheorem{lemma}{lmm}{Lemma}{\sl}
\addtheorem{claim}{claim}{Claim}{\sl}
\addtheorem{corollary}{corol}{Corollary}{\sl}
\addtheorem{pro}{prpstn}{Proposition}{\sl}
\addtheorem{problem}{problem}{Problem}{\sl}
\addtheorem{note}{rmrk}{Remark}{\normalfont}
\addtheorem{example}{example}{Example}{\normalfont}

\usepackage{paralist}

\title{On the number of autotopies of an $n$-ary qusigroup \\of order $4$%
\thanks{The results were reported in part at the International conference  
``Mal'tsev Meeting''  (Novosibirsk, 12--15 November 2013)}
\thanks{The work was supported 
by the RFBR grants 10-01-00616, 13-01-00463,
and 
by the program of fundamental scientific researches of the SB RAS No I.1.1., projects 0314-2016-0001, 0314-2016-0016.
}
}
\author{Еvgeny~V.~Gorkunov, Denis~S.~Krotov, Vladimir~N.~Potapov
\thanks{The authors are from the Sobolev Institute of Mathematics, pr. Akademika Koptyuga 4, Novosibirsk 630090, Russia; and from the Novosibirsk State University, Pirogova 2, Novosibirsk 630090, Russia. E-mail: evgumin@gmail.com,
krotov@math.nsc.ru,
vpotapov@math.nsc.ru.
}
}
\begin{document}
\maketitle
\begin{abstract}
An algebraic system from a finite set $\Sigma$ of cardinality $k$ and an $n$-ary operation $f$ invertible in each argument is called an $n$-ary quasigroup of order $k$. An autotopy of an $n$-ary quasigroup $(\Sigma,f)$ is a collection $(\theta_0,\theta_1,\ldots,\theta_n)$ of $n+1$ permutations of $\Sigma$ such that $f(\theta_1(x_1),\ldots,\theta_n(x_n))\equiv \theta_0(f(x_1,\ldots,x_n))$. We show that every $n$-ary quasigroup of order $4$ has at least $2^{[n/2]+2}$ and not more than $6\cdot 4^n$ autotopies. We characterize the $n$-ary quasigroups of order $4$ with $2^{(n+3)/2}$, $2\cdot 4^n$, and $6\cdot 4^n$ autotopies.
\end{abstract}

%=================================
%=================================
%=================================
\section{Introduction}\label{s:intro}

We denote by $ \Sigma $ the set of $ k $ elements  $ 0 $, $ 1 $, \ldots, $ k-1 $.
The Cartesian degree  $ \Sigma ^ n $
consists of all tuples of length $ n $ formed by elements of $ \Sigma $.
An algebraic system with support $ \Sigma $ and an $ n $-ary operation $ f \colon \Sigma ^ n \to \Sigma $
 invertible in each argument is called an \emph {$ n $ -ary quasigroup of order $ k $} (sometimes, for brevity,
an $ n $-quasigroup or simply a quasigroup).
The corresponding operation $ f $ is also called a quasigroup.

\emph {Isotopy} of the set $ \Sigma ^ {n + 1} $ is called
set $ \theta = (\theta_0, \theta_1, \ldots, \theta_n) $ permutations from the symmetric group $ S_k $
acting on $ \Sigma $.
The isotopy action on $ \Sigma ^ {n + 1} $
is given by the rule
$$ \theta \colon {x} \mapsto \theta ({x}) = (\theta_0 (x_0), \ldots, \theta_n (x_n)) \quad \mbox {for } x = (x_0, \ldots, x_n) \in \Sigma ^ {n + 1}. $$
To denote the isotopies and the permutations that constitute them, we will use the Greek alphabet,
and when writing their action on the elements of $ \Sigma $ we sometimes omit parentheses.

Two sets $ M_1, M_2 \subseteq \Sigma ^ {n + 1} $ are called \emph {isotopic} if
here exists an isotopy $ \theta $ such that $ \theta (M_1) = M_2 $. Two quasigroups $ f $ and $ g $
are called \emph{isotopic} if
for some isotopy $ \theta = (\theta_0, \theta_1, \ldots, \theta_n) $ it holds
\begin{equation} \label {eq:q-iso} g (x_1, \ldots, x_n) = \theta_0 ^ {- 1} f (\theta_1 x_1, \ldots, \theta_n x_n). \end{equation}
If $ f = g $, then any isotopy $ \theta $ for which (\ref {eq:q-iso}) holds is called an \emph {autotopy} of the quasigroup $ f $.
The \textit {autotopy group} $ \mathrm {Atp} (f) $ of a quasigroup $ f $ is the group
consisting of all autotopies of  $ f $ (the group operation is the composition).

A $ 2 $-quasigroup $ f $ with \emph{neutral} element $ e $ such that 
$ f (e, {a}) = f (a, {e}) = a $  for any $ a \in \Sigma $ is called a \emph{loop}. If a loop $ f $
satisfies the associative axiom $ f (x, f (y, z)) \equiv f (f (x, y), z) $, then we have a group.
It is known (see, for example, \cite{Belousov}),
that all $ 2 $-quasigroups of order $ 4 $ are isotopic to either the group $ Z_2 \times {Z_2} $ or the group $ Z_4 $.
So, quasigroups generalize groups, which illustrates
their algebraic nature.

At the same time, the concept of a quasigroup admits a purely combinatorial interpretation. By \textit {line}
in $ \Sigma ^ {n + 1} $, we mean a subset of $ n $ elements that are distinct exactly in one coordinate.
For a quasigroup $ f \colon \Sigma ^ n \to \Sigma $, the set
$ M (f) = \{(x_0, x_1, \ldots, x_n) \in \Sigma ^ {n + 1} \mid x_0 = f ({x}) \} $ will be called the \textit {code} of the
quasigroup $ f $. The term ``code'' is borrowed from the theory of error correcting codes,
in the framework of which the set $ M (f) $ is an MDS-code with distance $ 2 $
(an equivalent concept, also well known in combinatorics, is the Latin hypercube).
The quasigroup code is characterized as a subset $ \Sigma ^ {n + 1} $ of cardinality $ k ^ n $ intersecting
each line in exactly one element. This view allows us to see a quasigroup
from its combinatorial side.
We note that the codes of isotopic quasigroups are isotopic,
namely, it follows from (\ref{eq:q-iso}) that $ M (g) = \theta ^ {-1} (M (f)) $.

In this paper, we investigate autotopies of quasigroups of order $ 4 $. 
We establish  tight upper and lower
bounds on the order of the group of autotopies of such a quasigroup. 
In a way, it is natural 
that the richest group of autotopies turned out to be for the quasigroups called linear  with a structure close to group.
Also we characterize the quasigroups with minimum and pre-maximum 
(that is, text to the maximum)
 orders of the autotopy group.

The concept of an autotopy is a generalization of 
a more partial notion of ``automorphism'' and reflects in some
 sense the ``regularity'' or ``symmetry'' of a quasigroup as a combinatorial object.
The study of the transformations of the
space mapping the object on itself is a classic, 
but at the same time a difficult task,
considered in many areas of mathematics. 
The complexity of such problems is illustrated by Frucht's  theorem
\cite{Frucht:38} stating each finite group 
is isomorphic to the group of automorphisms of some graph, 
and also a similar result 
concerning perfect codes in coding theory \cite{Phelps:86}.

In coding theory, 
the group of automorphisms of a code is generated by the isometries of the metric
 space that sitabilize the code. There we find another example of the phenomenon 
that an object with group properties has the richest group of automorphisms.
In papers \cite{Mal:2000:aut.en,SolTop2000pit,SolTop2000en} it is shown that the binary Hamming code, a linear
perfect code, has the largest group of automorphisms among the binary perfect $1$-codes,
and its order at least twice exceeds the order of the automorphism group of any other binary $1$-perfect code of the same length.

The paper is organized as follows. Section \ref {s:def} provides basic definitions and statements.
In Section \ref {s:repres}, a representation of quasigroups necessary for further proof of the fundamental
results is given. 
Auxiliary statements on the autotopies of quasigroups are collected in Section~\ref{s:au}. 
A tight
lower bound on the number autotopies of a quasigroup of order $ 4 $ is proved in Section \ref{s:lb}. 
In Section \ref{s:min} we discuss the quasigroups 
with the smallest order of the autotopy group.
Finally, in Section~\ref {s:up}, an upper bound on the order of the autotopy group
quasigroup of order $ 4 $ is derived and it is proved that this bound
is attained only on the linear quasigroups.
We also establish the maximum order of the autotopy group of a nonlinear quasigroup of order $ 4 $ and prove
that it is attained only on isotopically transitive quasigroups, which were described in \cite{KP:2016:transitive}.

%=================================
%=================================
%=================================
\section{Notations and basic facts}\label{s:def}

For $ {x} = (x_1, \ldots, x_n) \in \Sigma ^ n $ and $ a \in \Sigma $, we put $ {x} _i ^ a = (x_1, \ldots, x_ {i-1} , a, x_ {i + 1}, \ldots, x_n) $.
The \emph {inverse} of an $ n $-quasigroup $ f $ in the 
$ i $-th argument is denoted by $ f ^ {\langle {i} \rangle} $; 
that is, for any $ {x} \in \Sigma^n $ and $ a \in \Sigma $, the equations
$ f ^ {\langle {i} \rangle} ({x} _i ^ a) = x_i $ and $ f({x}) = a $ are equivalent. Obviously, the inversion of an $ n $-quasigroup in any argument is an
$ n $-quasigroup. By the $0$-th argument of an $ n $-quasigroup $ f $, we mean the value of the function $ f (x_1, \ldots, x_n) $,
which, formally not being an argument of the operation $ f $ itself, is associated with the $ i $-th argument of its inverse $ f^{\langle {i} \rangle} $.

In this paper, we study the autotopies of quasigroups of order $ 4 $; so, below we assume
$ \Sigma = \{0,1,2,3 \} $.
A quasigroup $ f $ of order $ 4 $ is said to be \emph {semilinear} if there are
$ a_j, b_j \in \Sigma $, $ a_j \ne b_j $, $ j = 0,1, \ldots, n $, for which
\begin {equation} \label {eq:semilinear}
f (\{a_1, b_1 \} \times \ldots \times \{a_n, b_n \}) = \{a_0, b_0 \}.
\end {equation}
In this case, we also say that the quasigroup $ f $ is \textit {$ \{a_j, b_j \} $-semilinear
in the $ j $-th argument}, $ j = 0,1, \ldots, n $.
Note that if in the identity (\ref {eq:semilinear}) any two of the sets $ \{a_j, b_j \} $ are replaced by their complements in $ \Sigma $,
then the identity will remain true.
Thus, in every argument, a semilinear quasigroup is $ \{0,1 \} $-, $ \{0,2 \} $- or $ \{0,3 \} $-semilinear.
% [??? can simplify the entry: $ ab $-semilinear? or not ...]
If $ f $ is $ \{a, b \} $-semilinear in each of its arguments,
then we call it simply \textit {$ \{a, b \} $-semilinear.}

A quasigroup $ f $ is \textit {linear} if
in each of its arguments it is $ \{a, b \} $-semilinear for any $ a, b \in \Sigma $.

Each $ 2 $-quasigroup is isotopic to one of the two quasigroups $ \oplus $, $ +_4 $ with the value tables
\begin{equation} \label{eq:Z22,Z4}
\begin{array}{c|cccc}
    \oplus & 0 & 2 & 1 & 3\\
    \hline
    0 & 0 & 2 & 1 & 3\\
    2 & 2 & 0 & 3 & 1\\
    1 & 1 & 3 & 0 & 2\\
    3 & 3 & 1 & 2 & 0
\end{array}, \qquad
\begin{array}{c|cccc}
    +_4 & 0 & 2 & 1 & 3\\
    \hline
    0 & 0 & 2 & 1 & 3\\
    2 & 2 & 0 & 3 & 1\\
    1 & 1 & 3 & 2 & 0\\
    3 & 3 & 1 & 0 & 2
\end{array}.
\end{equation}
The quasigroups $ (\Sigma, \oplus) $ and $ (\Sigma, + _ 4) $ are the groups $ \mathbb {Z} _2 \times \mathbb {Z} _2 $ and $ \mathbb {Z} _4 $, respectively. 
% (Recall that the group is an associative $ 2 $-quasigroup with neutral element.)
\begin {note} In the value tables (\ref {eq:Z22,Z4}), the elements $ 0 $, $ 2 $, $ 1 $, $ 3 $ are not ordered lexicographically, in the usual sense. 
%but in increasing order in the group $ \mathbb {Z} _4 $.
With the given ordering, it is easier to observe the semilinear structure of the given group.
In the future, similarly, the table of values of a $ \{0,2 \} $-semilinear $ n $"=quasigroup is convenient to be thought as an $ n $-dimensional array $ 4 \times \ldots \times 4 $
which is divided into $ n $-dimensional subarrays 
$ 2 \times \ldots \times 2 $ filled with two values $ 0 $, $ 2 $ or $ 1 $, $ 3 $.
\end {note}
The quasigroup $ \oplus $ iterated $ n-1 $ times will be denoted by $ l_n $; that is,
$$ l_n(x_1,\ldots,x_n) = x_1\oplus\ldots\oplus x_n. $$

\begin{lemma}[\cite{PotKro:asymp}]\label{l:linear}
(i) All linear $ n $-quasigroups are isotopic to $ l_n $.
(ii) If the $ n $-quasigroup is simultaneously $ \{0,1 \} $- and $ \{0,2 \} $"=semilinear in some argument, then it is linear.
\end{lemma}

\begin{lemma}\label{l:isemilinear}
If an $ n $-quasigroup $ f $ is $ \{a, b \} $"=semilinear with respect to the $ i $-th argument for some
$ i \in \{0, ..., n \} $ and $ \theta = (\theta_0, \ldots, \theta_n) $ is an isotopy, then the $ n $-quasigroup $ \theta ( f) $ is
$ \{\theta_i^{- 1} (a), \theta_i ^ {- 1} (b) \} $-semilinear with respect to the $ i $-th argument.
\end{lemma}

A quasigroup $ f $ of arity $ n $ is said to be \emph{reducible} if for some integer
$ m, \, 2 \leqslant m <n $, and the permutation $ \sigma \in {S_n} $ holds, the representation $ f $ in the form of a repetition-free
compositions
\begin{equation}
%\label{eq:composition}
f(x_1,\ldots,x_n)= h(g(x_{\sigma(1)},\ldots,x_{\sigma(m)}),x_{\sigma(m+1)},\ldots,x_{\sigma(n)})
\end{equation}
(repetition-free means that each variable occurs only once in the right side).
Without loss of generality, we can assume that the quasigroup $ g $ is irreducible.

In~\cite{KroPot:4}, a description of quasigroups of order $ 4 $ is obtained in terms of semilinearity and reducibility.

\begin{theorem}
\label{th:top}
Every $n$-ary quasigroup of order $ 4 $ is reducible or semilinear.
\end{theorem}

%=================================
%=================================
%=================================
\section{The representation of quasigroups}\label{s:repres}

\hspace*{\parindent} According to Theorem~\ref{th:top}, a non-semilinear quasigroup
can be represented as a repeatition-free composition of two or more semilinear quasigroups (some of the composed quasigroups
can coincide with each other or be linear). A representation of the quasigroup $ f $
in the form of a repeatition-free composition of 
quasigroups of arity greater than $ 1 $ 
will be called a \textit{decomposition} of $ f $. Note that the quasigroup
may have several decompositions. In the simplest case, the quasigroup represents its
own trivial decomposition.

In the following, we will use a graphical representation of the decomposition of a quasigroup in the form of a labeled tree.
The inner vertices of this tree (the degree of which is not less than $ 3 $) will be called \textit {nodes} and denoted by
characters% of quasigroups included in the decomposition
$ u$, $v$, $ w $, with or without indices;
and the leaves (vertices of degree $ 1 $) are the symbols of the variables $ x_1$, $x_2$, \ldots, $y$, $z $.
% associated with the corresponding arguments of the quasigroup.
The edge incident to a leaf of the tree is called a \textit{leaf} edge. 
The remaining edges are called
\textit{inner}.

Firstly, we define recursively the \textit{root tree} $ T (S) $ of a decomposition $ S $
(the notion of the root decomposition tree is introduced as an auxiliary term
 to define the decomposition tree and will not be used after that definition).
\par
1) A variable $ x_i $ is associated with the tree consisting of one vertex of degree $ 0 $,
being the root and labelled by the variable $ x_i $  itself.
% Квазигруппе $f(x_1,x_2)$ соответствует дерево, состоящее из корня $u$ с меткой $f$ and двух листьев $x_1$ and $x_2$, смежных с корнем.
\par
2) 
Let a decomposition of $ S $ be of the form $ S = h (S_1, \ldots, S_n) $. If the decompositions and/or variables $ S_1$, \ldots, $S_n $
correspond to the root trees $ T_1$, \ldots, $T_n $, respectively, then we build the tree $ T (S) $ as follows.
Define a new vertex $ u $ as the root of the tree and assign the label $ h $ to it. 
Consistently connect
vertex $ u $ with the roots of the trees $ T_1$, \ldots, $T_n $. 
The root of the tree
$ T_j $, $ j \in \{1, ..., n \} $, is considered as the $j$-th neighbor of $ u $. 
On the other hand, the vertex $ u $ is considered as the $0$-th neighbor of the root of the tree $ T_j $.

By the \textit{decomposition tree} (without ``root''),
we call the tree obtained by connecting 
the leaf $x_0$ as the $0$-th neighbor to the root of the tree $T(S)$.
The decomposition tree of the decomposition $S$ is denoted by $T_0(S)$.
The leaf $x_0$ corresponds to the $0$-th argument, i.e., to the value
of the quasigroup represented by the decomposition $S$.

The tree $T_0(S)$ of the decomposition $S$ of $f$ can be treated
as the decomposition tree for the code $M(f)$.
It is important to understand that only the enumeration 
of the leaves and of the neighbors of every vertex defines 
which arguments are independent for the quasigroup $f$ 
and for every element of the decomposition.
Changing this enumeration, we can get the decomposition tree
for the inverse of $f$ in any argument.
Namely, to get a decomposition for $f^{<i>}$,
it is sufficient for every inner node of the path from
 $x_0$ 
to $x_i$, swap the labels of the two neighbors, replace the node label to the corresponding inverse, and finally swap the labels $x_i$ and $x_0$.
The order defined on the neighbors 
of every node uniquely determines 
the order of the arguments of the quasigroups in the decomposition and of the represented quasigroup.
We note that as the autotopy groups 
of a quasigroup and its inverses are isomorphic,
from the point of view of the questions considered in the current research, it is not necessary to remember all the time which of the arguments is the $0$-th one;
so, the $0$-th argument will not be emphasized in the most of considerations.
%Лист $x_0$ позволяет восстановить корень дерева разложения, а 

For a decomposition and the corresponding decomposition tree, define the operation of \emph{merging}.
Assume that a decomposition $S$ 
contains the fragment 
\begin{equation}
\label{eq:fragment}
	f_1(S_1,\ldots,S_{i-1},f_2(S_i,\ldots,S_{i+n_2-1}),S_{i+n_2},\ldots,S_{n_1+n_2-1}),
\end{equation}
where $S_1$, \ldots, $S_{n_1+n_2-1}$ are some decompositions and the $n_1$-quasigroup $f_1$ 
and $n_2$-quasigroup $f_2$ 
%на множестве $\Sigma^{n_1+n_2-1}$, 
satisfy the identity 
\begin{equation}
\label{eq:quasifragment}
 g(x_1,\ldots,x_{n_1+n_2-1})\equiv f_1(x_1,\ldots,x_{i-1},f_2(x_i,\ldots,x_{i+n_2-1}),x_{i+n_2},\ldots,x_{n_1+n_2-1})
\end{equation}
for some $(n_1+n_2-1)$-quasigroup $g$.
%and попарно различных переменных $x_1,\ldots,x_{n_1+n_2-1}$. 
The result of merging $f_1$ and $f_2$
in $S$ is defined as the decomposition $\widetilde{S}$
obtained from $S$ by replacing the fragment~\eqref{eq:fragment} 
by $g(S_1,\ldots,S_n)$. 
Note that we consider a concrete occurrence  of~\eqref{eq:fragment} in $S$
(in general a fragment can occur more than one time).

Respectively, in the decomposition tree $T_0(S)$, the adjacent nodes $u$ and $v$ 
labeled by $f_1$ and $f_2$ are merged as follows. This pair of nodes is replaced by a new node~$w$ labeled by $g$, whose neighbors are the neighbors of the removed nodes $u$ and $v$ (except  $u$ and $v$ themselves). 
The neighbors of $w$ are assigned the numbers
$0$, \ldots, $n_1+n_2-1$ consequently in the following order. 
At first, the neighbors of $u$ with numbers 
$0$, \ldots, $i-1$ are assigned (in the same order); next,
 the neighbors of  $v$ with the numbers $1$, \ldots, $n_2$ are assigned; then, the remaining neighbors of $u$ with the numbers $i+1$, \ldots, $n_1$ are assigned.
 The result of the described merging is the decomposition tree~$T_0(\widetilde{S})$. Trivially, we have the following fact.

\begin{lemma}\label{l:styag}
Merging does not alter the quasigroup represented by the decomposition.
\end{lemma}

We call a decomposition (and its tree) \textit{semilinear} if all involved quasigroups are semilinear.

In a decomposition tree, consider two neighbor nodes $u$, $v$ with labels $f_1$, $f_2$, respectively. Assume that $u$ is the $0$-th neighbor of $v$
and $v$ is the $i$-th neighbor $u$. We call the nodes $u$ and $v$
\emph{coherent} if for some $a,b\in\Sigma$ the quasigroup $f_1$ 
is $\{a,b\}$-semilinear in the $i$-th argument and $f_2$ 
is $\{a,b\}$-semilinear in the $0$-th argument.

\begin{lemma}\label{l:soglas}
Merging two coherent nodes in a semilinear tree results
in a semilinear tree.
\end{lemma}

\begin{proof} Let the quasigroups $ f_1 $ and $ f_2 $ of arity $ n_1 $ and $ n_2 $ correspond to coherent
nodes in the decomposition tree, and (\ref{eq:quasifragment}) holds 
for some  $ (n_1 + n_2-1 )$-ary quasigroup $ g $.
To prove the lemma, it suffices to verify that the quasigroup $ g $ 
is semilinear, 
which is straightforward from (\ref{eq:quasifragment}) and
the definition of a semilinear quasigroup.
\end{proof}

We call a semilinear decomposition (and its tree) 
 \emph{proper} if there are no pairs of coherent nodes in the decomposition tree.

\begin{lemma}\label{l:prav}
Every quasigroup of order $4$ has a proper decomposition.
\end{lemma}

\begin{proof} By Theorem~\ref{th:top}, every $ n $-ary quasigroup of order $ 4 $ has a semilinear decomposition.
Since there are no more than $ n-1 $ nodes in the decomposition tree, successively merging pairs of coherent nodes,
we obtain a required decomposition in at most $ n-2 $ steps, 
\end{proof}

\begin{note}
In general, a proper decomposition is not unique and depends on the order of merging.
The simplest example of the decomposition that can be merged in two ways is $f(g(h(x_1,x_2),x_3),x_4)$, where $f$ and $g$ are $\{0,1\}$-semilinear quasigroups, $g$ and $h$  are $\{0,2\}$-semilinear quasigroups, and $f$, $h$ are not linear, in contrast to $g$.
A proper decomposition of a nonlinear quasigroup does not involve 
linear quasigroups, because a node labeled by a linear quasigroup is coherent with each if its neighbors.
\end{note}

Let  $S$ be some  decomposition with the tree $T_0(S)$, 
with the set of edges $E$. 
An \emph{isotopy} of the decomposition is a collection $\theta=(\theta_e)_{e\in{E}}$ of permutations of~$\Sigma$,
acting on $S$ as follows. 
If a node of~$T_0(S)$ has a label~$f_i$ 
and $e_j$, $j=0,1,\ldots,n_i$, is the $j$-th edge incident to this node,
then $f_i$ is replaced by $f'_i$, where $f'_i$ is the quasigroup
defined by
\begin{equation}\label{eq:isot_i}
f'_i(x_1,\ldots,x_{n_i}) = \theta_{e_0}^{-1}f_i(\theta_{e_1} x_1,\ldots,\theta_{e_{n_i}} x_{n_i}).
\end{equation} 
As a result, we get the tree of some decomposition denoted $\theta(S)$
and called isotopic to $S$. The following is straightforward.

\begin{lemma}\label{l:isotop}
Isotopic decompositions represent isotopic quasigroups. 
More precisely, if $\theta$ is an isotopy connecting decompositions
of quasigroups $f$ and $f'$, then
$$x_0=f'(x_1,\ldots,x_{n}) = \theta_{e_0}^{-1}f(\theta_{e_1}(x_1),\ldots,\theta_{e_{n}}(x_{n})),$$
where $e_j$ is the edge incident to the leaf $x_j$, $j=0,1,\ldots,n$.
\end{lemma}

An \emph{autotopy} of the decomposition $S$ is an isotopy $\theta$ 
such that $\theta(S)=S$. 
The \emph{support}
of an  autotopy is the set of edges corresponding to non-identity permutations.
We call a proper decomposition (and its tree) \emph{reduced}
if every involved quasigroup is
$\{0,1\}$- or $\{0,2\}$-semilinear.

\begin{lemma}\label{l:dobr}
For every quasigroup of order $4$,
there is an isotopic quasigroup with a reduced decomposition.
\end{lemma}

\begin{proof} 
Consider an $ n $-quasigroup $ f $ 
and construct an isotopic quasigroup with a reduced
decomposition. We start with a proper decomposition $ S $ of $ f $, 
which exists by Lemma~\ref{l:prav}.
Since the decomposition tree $ T_0 (S) $ is a bipartite graph,
its vertices are divided into two independent parts; 
the vertices of one part are called even, the other are odd.

Let us find an isotopy $ \theta $ 
such that the odd nodes of $ \theta (S) $ are $ \{0,1 \} $-semilinear,
while the even nodes are $ \{0,2 \} $-semilinear. 
To do this, 
we define the permutation~$ \theta_e $ for every edge $ e $
in the tree $ T_0 (S) $.

Consider two cases. Firstly, let $ e $ connect two nodes, 
an odd one with a label~$ g $ and an even one labeled by $ h $. 
Suppose that $ g $ is the $ i $-th neighbor of $ h $, 
which in turn is the $0$-th neighbor of $ g $.
Note that if  $ g $  is $ \{0, a \} $-semilinear 
in the $0$-th argument
and $ h $ is $ \{0, b \} $-semilinear in the $ i $-th, 
then $ a \ne {b} $, because the decomposition
$ S $ is proper. 
In this case, we set 
$ \theta_e (0) = 0 $, 
$ \theta_e (1) = a $, 
$ \theta_e (2) = b $,
$ \theta_e (3) \in \{1,2,3 \} \setminus \{a, b \} $.

Now we turn to the other case 
when $e$ connects a node labeled by $ g $ 
and its $i$-th neighbor, a variable $ x $.
Suppose $ g $ is $ \{0, a \} $-semilinear in the $ i $-th argument. 
Then we set $ \theta_e = (1a) $
if the node is in the odd part of~$ T_0 (S) $, 
and $ \theta_e = (2a) $ if the node is in the even part.

Consider the action of the constructed isotopy on the decomposition $ S $. 
A node of the decomposition tree $ T_0 (S) $ labeled
 by $ f_i $ will get the label  $ f'_i $ 
 (see (\ref {eq:isot_i})) in the tree $ T_0 (\theta (S)) $. 
 Suppose $ f_i $ is $ \{0, a \} $-semilinear in the $ j $-th argument. Then, by Lemma~\ref {l:isemilinear}, the quasigroup $ f'_i $ is $ \{0,1 \} $-semilinear in the $ j $-th argument
if $ f_i $ is an odd node, and $ \{0,2 \} $-semilinear in the $ j $-th argument if $ f_i $ is an even node.

Thus, the decomposition $ \theta(S) $ is proper by the definition. 
By Lemma~\ref{l:isotop}, the quasigroup represented
is isotopic to the original quasigroup $ f $.
\end{proof} 

%=================================
%=================================
%=================================
\section{Autotopies of quasigroups}\label{s:au}

\hspace*{\parindent} Let $\pi=(\pi_0,\ldots,\pi_m)$ and $\tau=(\tau_0,\ldots,\tau_{n-m+1})$ be isotopies. 
If $\pi_0=\tau_1$, then we define
$$\pi\dot{\otimes}\tau=(\tau_0,\pi_1,\ldots,\pi_m,\tau_2,\ldots,\tau_{n-m+1}).$$

Let us consider the $n$-ary quasigroup $f$
obtained as the composition of an $m$-quasigroup~$g$ and $(n-m+1)$-quasigroup $h$:
$$f(x_1,\ldots,x_n)= h(g(x_1,\ldots,x_m),x_{m+1},\ldots,x_n).$$
We define the action of the operation $\dot{\otimes}$ on the autotopy groups of $g$ and $h$ as follows:
$$
	\mathrm{Atp}(g)\dot{\otimes}\mathrm{Atp}(h)=
	\{\pi\dot{\otimes}\tau\mid \pi=(\pi_0,\ldots,\pi_m)\in\mathrm{Atp}(g),\tau=(\tau_0,\ldots,\tau_{n-m+1})\in\mathrm{Atp}(h), \pi_0=\tau_1\}.
$$
We restrict ourselves by considering quasigroups of order $4$ only;
however, the next lemma holds for any other order as well.

\begin{lemma}
\label{lem:composition}
If $f$ is an $n$-quasigroup represented as the composition 
$$f(x_1,\ldots,x_n)= h(g(x_1,\ldots,x_m),x_{m+1},\ldots,x_n),$$ then 
$$\mathrm{Atp}(f)=\mathrm{Atp}(g)\dot{\otimes}\mathrm{Atp}(h).$$
\end{lemma}

\begin{proof}
Obviously, $\mathrm{Atp}(g)\dot{\otimes}\mathrm{Atp}(h)\leqslant\mathrm{Atp}(f)$. To prove the reverse, consider an autotopy 
$\theta=(\tau_0,\pi_1,\ldots,\pi_m,\tau_2,\ldots,\tau_{n-m+1})\in\mathrm{Atp}(f)$. Let us show that there exists a permutation
$\pi_0=\tau_1\in{S_4}$ such that 
$\pi=(\pi_0,\ldots,\pi_m)\in\mathrm{Atp}(g)$, 
$\tau=(\tau_0,\ldots,\tau_{n-m+1})\in\mathrm{Atp}(h)$,
and $\theta=\pi\dot{\otimes}\tau$.

Note that if such permutation $\pi_0$ exists, then 
it is uniquely defined by the permutations $\pi_1$, \ldots, $\pi_m$, 
because for every tuple from $\Sigma^n$ the quasigroup $g$ posseses only one value. 
Moreover, if we put 
$\pi_0=\tau_1$, then $\pi\in\mathrm{Atp}(g)$ if and only if $\tau\in\mathrm{Atp}(h)$. 
Indeed, the relation $\pi\in\mathrm{Atp}(g)$, by the definition,
means that the equations
\begin{multline*}
	x_0=h(g(x_1,\ldots,x_m),x_{m+1},\ldots,x_n) \text{ and }\\
	x_0=h(\pi_0 g(\pi_1^{-1}x_1,\ldots,\pi_m^{-1}x_m),x_{m+1},\ldots,x_n)
\end{multline*}
are equivalent. 
Applying the autotopy $\theta\in\mathrm{Atp}(f)$, we get for any tuple
$(x_1,\ldots,x_n)\in{\Sigma^n}$:
\begin{multline*}
	x_0=h(g(x_1,\ldots,x_m),x_{m+1},\ldots,x_n)=\\
	=\tau_0^{-1}h(\pi_0g(x_1,\ldots,x_m),\tau_2 x_{m+1},\ldots,\tau_{n-m+1}x_n).
\end{multline*}
The last equality implies that for any $(t,x_{m+1},\ldots,x_n)\in{\Sigma^{n-m+1}}$ it holds
$$
	h(t,x_{m+1},\ldots,x_n)=\tau_0^{-1}h(\pi_0 t,\tau_2 x_{m+1},\ldots,\tau_{n-m+1}x_n).
$$
That is, for $\tau_1=\pi_0$ we have $\tau\in\mathrm{Atp}(h)$.

So, it remains to show that there exists a permutation $\pi_0\in{S_4}$ such that
$\pi\in\mathrm{Atp}(g)$. 
Taking into account that $\theta\in\mathrm{Atp}(f)$, we can write that for every
${(x_1,\ldots,x_m)}\in{\Sigma^m}$ it holds
\begin{equation}
\label{eq:zeroes}
	x_0=h(g({x}),0,\ldots,0)=\tau_0^{-1}h(g(\pi_1 x_1,\ldots,\pi_m x_m),\tau_2(0),\ldots,\tau_{n-m+1}(0)).
\end{equation}
Trivially, the $1$-quasigroups
$$q_1(s)=h(s,0,\ldots,0),\quad q_2(t)=\tau_0^{-1}h(t,\tau_2(0),\ldots,\tau_{n-m+1}(0))$$ 
are permutations of $\Sigma$. So, (\ref{eq:zeroes}) can be rewritten as follows:
$$g({x})=q_1^{-1} ( q_{2}( g  (\pi_1x_1,\ldots,\pi_mx_m))) .$$
Defining $\pi_0(\cdot)=q_2^{-1}(q_1(\cdot)) $, we have $(\pi_0,\ldots,\pi_m))\in\mathrm{Atp}(g)$.
\end{proof}

As follows from Lemma~\ref{lem:composition} and the results of the previous section,
studying the autotopy group of a quasigroup of order $4$ we can regard that it is represented 
as the repetition-free composition of quasigroups, each of which is 
$\{0,a\}$-semilinear for some $a\in\Sigma$, but not linear.

In the remaining part of this section, we prove three lemmas on minimum autotopy groups
of semilinear quasigroups.
In the description of autotopies, it is convenient to use the following notation.

For a nonlinear $\{0,a\}$-semilinear quasigroup
(and the corresponding nodes of decomposition trees), $a\in\Sigma\setminus\{0\}$, 
the permutation $(0a)(bc)$, where $\{b,c\}=\Sigma\setminus\{0,a\}$, is called 
the \emph{native involution}, 
and the permutations $(0b)(ca)$ and $(0c)(ab)$ are called  \textit{foreign involutions}.
Each of the transpositions $(0a)$ and $(bc)$ forming the native involution $(0a)(bc)$
is called a  
\emph{native transposition} of the semilinear quasigroup (node).
The two cyclic permutations $(0bac)$ and $(0cab)$ whose square is the native involution $(0a)(bc)$
are called the \emph{native cycles} of the semilinear quasigroup (node).

\begin{lemma}\label{l:2rodn}
The following isotopies belong to the autotopy group of a
$\{0,a\}$-semilinear $n$-ary quasigroup $f$, $a\in\Sigma\setminus\{0\}$.
\par
{\rm (i)} An isotopy consisting of two native involutions and $n-1$ identity permutations, in an arbitrary order.\par
{\rm (ii)} An isotopy from $n+1$ native transpositions different from $(0a)$ if $f(\{0,a\}^n)=\{0,a\}$,
or from $n+1$ native transpositions exactly one of which is $(0a)$
if $f(\{0,a\}^n)=\Sigma\backslash\{0,a\}$.
\end{lemma}
\begin{proof} Without loss of generality we assume $a=1$.
The identity (\ref{eq:semilinear}) holds for any $\{a_i,b_i\}$ from $\{\{0,1\},\{2,3\}\}$, $i=1,\ldots,n$
(the pair $\{a_0,b_0\}$ is uniquely defined from the other pairs and also coincides with
$\{0,1\}$ or $\{2,3\}$).

(i) Applying the native involution $ (01) (23) $ in one of the arguments changes the values of the quasigroup in all points,
but at the same time leaves the sets $ \{a_1, b_1 \} \times \ldots \times \{a_n, b_n \} $ with the above restrictions in place.
It follows from (\ref{eq:semilinear}) that the values of the quasigroup also change in accordance with the native involution.
When applying the native involution in some other argument, we again obtain the original quasigroup.

(ii) Let $f(\{0,1\}^n)=\{0,1\}$. 
Consider an arbitrary tuple $ (x_1, \ldots, x_n) $ of values of the arguments and the value $ x_0 $ of the quasigroup on this tuple.
Among $ x_0 $, $ x_1 $, \ldots, $ x_n $, an even number of values belong to $ \{2,3 \} $. Thus, applying successively the transposition $ (23) $ to each of the arguments,
we  change the value of the quasigroup an even number of times, 
and the changes do not take the value in a partial point beyond the pair $ \{0,1 \} $ 
or the pair $\{2,3 \} $.
As a result, we get that after applying all transpositions, the value of the quasigroup has not changed. The case $ f (\{0,1 \} ^ n) = \{2,3 \} $ is treated similarly.
\end{proof}

\begin{lemma}\label{l:aut3}
Assume that an $\{a,b\}$-semilinear binary quasigroup $q$ of order $4$ is not linear.
Let $\xi$ be the corresponding native involution. 
The autotopy group of~$q$ consists of the following transformations.\par
{\rm (i)}\;The autotopies $(\mathrm{Id},\xi,\xi),$ $(\xi,\mathrm{Id},\xi),$ $(\xi,\xi,\mathrm{Id}),$ and the identity autotopy.\par
{\rm (ii)}\;The autotopies  $(\tau_0,\tau_1,\tau_1),$ $(\tau_1,\tau_0,\tau_1),$ $(\tau_1,\tau_1,\tau_0),$ $(\tau_0,\tau_0,\tau_0),$ 
where $\tau_0$,  $\tau_1$ are the two distinct native transpositions; the choice of $\tau_0$ 
is unique for the given $q$. \par
{\rm (iii)}\;The autotopies  $(\xi',\varphi_1,\varphi_2),$ $(\varphi_1,\xi'',\varphi_2),$ $(\varphi_1,\varphi_2,\xi'''),$ 
where $\varphi_1$, $\varphi_2$ is an arbitrary pair of native cycles, for which the permutations
 $\xi',\xi'',\xi'''\in\{\mathrm{Id},\xi\}$ are uniquely defined.\par
{\rm (iv)}\;The autotopies  $(\tau',\psi_1,\psi_2),$ $(\psi_1,\tau'',\psi_2),$ $(\psi_1,\psi_2,\tau'''),$ 
where $\psi_1$, $\psi_2$ is an arbitrary pair of foreign involutions, for which the native transpositions $\tau'$, $\tau''$, $\tau'''$ are uniquely defined.
\end{lemma}
\begin{proof}
It can be directly checked that each of the presented isotopies is an autotopy of $q$.
To do this, it is sufficient to consider the $\{0,2\}$-semilinear quasigroup $+_4$ 
(see Example~\ref{ex:1} below), because all quasigroups satisfying the hypothesis of the lemma are isotopic to  $+_4$.
It is easy to see that the set of presented autotopies is closed under the composition;
that is, this set forms a group.
 
 The completeness is checked numerically.
 There are
 $4$ autotopies of each of the type (i), (ii) and $12$ autotopies of each of the type (iii), (iv);
 totally we have $32$ autotopies.
 On the other hand, we can bound the number of autotopies from the upper side.
 It follows from the nonlinearity of $+_4$ and Lemma~\ref{l:isemilinear} that for an
 arbitrary autotopy $(\psi_0,\psi_1,\psi_2)$ 
 each of the permutations $\psi_0$, $\psi_1$, $\psi_2$ maps $\{0,2\}$ to $\{0,2\}$ or $\{1,3\}$.
 There are $8$ ways to choose $\psi_1$ meeting this condition, and $8$ ways for $\psi_2$;
 by the definition of a quasigroup, $\psi_0$ is determined uniquely from $\psi_1$ and $\psi_2$.
 Moreover, it is easy to check that there is no autotopy with
  $\psi_1=\mathrm{Id}$ and $\psi_2=(01)$.
 It follows that the order of the autotopy group is less than $64$;
 hence, this group coincides with the group from the autotopies (i)--(iv). 
\end{proof}

\begin{example}\label{ex:1}
Consider the binary quasigroup $+_4$ defined in (\ref{eq:Z22,Z4}).
The permutation $ (02) (13) $ is the native involution for $ q $; 
the permutations $ (02) $ and $ (13) $ are
the native transpositions for $ q $, and $ (0123) $, $ (0321) $ are the native cycles. The autotopy group of $ q $ is generated by the following (strictly speaking, redundant) set of autotopies:
\par
(i) $((02)(13),(02)(13),\mathrm{Id})$, $((02)(13),\mathrm{Id},(02)(13))$, $(\mathrm{Id},(02)(13),(02)(13))$;\par
(ii) $((13),(13),(13))$;\par
(iii) $(\mathrm{Id},(0123),(0321))$;\par
(iv) $((13),(03)(12),(01)(23))$, $((02),(01)(23),(01)(23))$, $((02),(03)(12),(03)(12))$.
\end{example}

Thus, we know the group of autotopies of the unique, up to isotopy, nonlinear binary quasigroup.
In addition, we need examples of semilinear $ 3 $- and $ 4 $-ary quasigroups with the minimal group of autotopies.
We define the $ n $-ary quasigroup $ l_n ^ \bullet $ by the identity

$$l_n^\bullet(x)= \begin{cases}
                    l_n(x)\oplus 2,& \mbox{if $x\in\{0,2\}^n$,}  \cr l_n(x), & \mbox{if $x\not\in\{0,2\}^n$.} 
                       \end{cases}
$$

\begin{lemma}\label{l:aut4}
 If $n\ge 3$ then the autotopy group of $l_n^\bullet$ 
 is generated by the autotopies from Lemma~\ref{l:2rodn} and has the order $2^{n+1}$.
\end{lemma}

We prove Lemma~\ref{l:aut4} for any $n$; however, we note that only
the cases $n=3$ and $n=4$ will be used in the further discussion.
For these cases, the statement of Lemma~\ref{l:aut4} can be checked directly.
\begin{proof}
 Obviously, the autotopies from Lemma~\ref{l:2rodn} have order $ 2 $, commute and are linearly independent; whence the order of the group generated by them follows.
 
 The code $ M (l_n) $ of the quasigroup $ l_n $ is a $ 2n $-dimensional affine subspace of the vector space over the field GF$(2)$ of two elements with addition $\oplus$ and trivial multiplication by $0$ and $1$.
 
The code $M(l_n^\bullet)$ of $l_n^\bullet$ differs from the affine subspace 
$M(l_n)$ in the $2^n$ vertices of the set $B_n$, where 
$B_n=M(l_n^\bullet)\backslash M(l_n) = \{(l_n^\bullet (x),x)\mid x\in\{0,2\}^n\}$.
 Moreover, $M(l_n)$ is a unique closest (in the sense above) to $M(l_n^\bullet)$ affine subspace,
 because any other affine subspace of the same dimension differs from $M(l_n)$ 
 in at least $2^{2n-1}\ge 4\cdot 2^n$ vertices.
 Under the action of an autotopy of $l_n^\bullet$, the code $M(l_n^\bullet)$ 
 is mapped to itself (by the definition), while $M(l_n)$ is mapped to an affine subspace
 (indeed, it is easy to see that any permutation of $\Sigma$ 
 is an affine transformation over GF$(2)$), 
 which is also closest to $M(l_n^\bullet)$. 
 It follows that an autotopy of $l_n^\bullet$ is necessarily an autotopy of $l_n$. 
 Moreover, it also follows that under the action of such autotopy the set $B_n$
 (the difference between the codes of $l_n^\bullet$ and $l_n$) is mapped to itself.
 In particular, every permutation of that autotopy stabilizes the set $\{0,2\}$, i.e., is one of
 $\mathrm{Id}$, $(02)$, $(13)$, $(02)(23)$.
 As it follows from the description of the autotopy group of $l_n$ in Section~\ref{s:up}, 
 all such autotopies are combinations of the autotopies from Lemma~\ref{l:2rodn}.
\end{proof}

%=================================
%=================================
%=================================
\section{A lower bound and quasigroups attaining it}\label{s:lb}
\subsection{The estimation}\label{s:lb1}

In this section, we consider an arbitrary quasigroup of order $4$ and prove 
a sharp lower bound for the order of its autotopy group. In particular, 
the autotopy group of a semilinear quasigroup is rather large. For a reducible 
quasigroup $f$, we show that the nodes of its decomposition tree $T_0(f)$ can be grouped 
into subsets, which we call bunches. Each bunch in $T_0(f)$ consists of nodes 
of the same parity, i.~e. it does not contain any adjacent nodes of the tree 
$T_0(f)$. A current subgroup of the autotopy group $\mathrm{Atp}(f)$ corresponds 
to each bunch, and the subgroups corresponding to different bunches are independent.

We now introduce additional notation and definitions concerning the representation 
of quasigroups in a form of a decomposition tree. Let $f$ be an $n$-ary quasigroup 
of order $4$ with a reduced decomposition $S$ and the decomposition tree $T=T_0(S)$.

$\bullet$\;Let $N=n+1$ denote the number of leaves in the tree $T$, and let $V$ be 
the number of nodes in $T$.

$\bullet$\;A \emph{bald} node is an inner vertex $u$ of the tree $T$ without leaves 
among the neighbors of $u$. Let $E$ equal the number of bald nodes in $T$.

$\bullet$\;A \emph{bridge} node, or simply \emph{bridge}, is a vertex $u$ of degree $3$ 
in the tree $T$ that is adjacent to exactly one leaf. The leaf adjacent to the bridge $u$ 
is called a \emph{bridge} leaf. Let $B$ equal the number of bridges in $T$.

$\bullet$\;A \emph{fork} is a vertex $u$ of degree $3$ in the tree $T$ that is adjacent 
to exactly two leaves. Let $F$ equal the number of forks in $T$.

$\bullet$\;By $G(T)$, we denote the graph with the set of nodes of the tree $T$ taken as 
the vertex set. Two vertices are adjacent in the graph $G(T)$ if the corresponding nodes are 
adjacent to the same bridge in $T$. It is easy to see that $G(T)$ is a forest.

$\bullet$\;A \emph{bunch} 
is a connected component of $G(T)$. 
Let $\mathit\Gamma$ equal the number of bunches in $T$.
 
$\bullet$\;For a bunch $G$ in $G(T)$, 
a leaf $x$ of the tree $T$ belongs to the \emph{leaf set} 
of $G$ if $x$ is adjacent to some node of $T$ included in $G$. 
A bunch of the graph $G(T)$ 
is called \emph{bald} if its leaf set is empty. Let $L$ equal the number of bald bunches in $T$.

It is worth to note that a bridge providing a corresponding edge in a bunch $G$ 
does not belong to $G$ as its vertex. The bridge being a node is contained 
in another bunch which differs from $G$. In addition, all bridges providing 
the edges of the bunch $G$ belong to one of two parts of the bipartite graph $T$ 
while the nodes of $G$ pertain to the other part of $T$.

For example, consider the decomposition tree designed in the Figure~\ref{fig1}. 
There are one bald node $\iota$, five bridges 
$\gamma$, $\delta$, $\zeta$, $\eta$, $\theta$, 
and one fork $\beta$. 
The nodes form seven bunches, namely 
$\{\alpha,\beta,\varepsilon,\eta,\iota\}$, 
$\{\gamma\}$,
$\{\delta\}$,
$\{\zeta,\theta\}$, 
$\{\kappa\}$,
$\{\lambda\}$,
$\{\mu\}$.
\tikzset{x=0.6cm,y=0.6cm,
	nodd/.style={rectangle,fill=white,draw=black,minimum size=0.5cm,rounded corners=4pt},
	leaf/.style={circle,fill=white,draw=black},
}
\begin{figure}[h]
$$
\begin{tikzpicture}
\draw (0,0) node[nodd,fill=red!50]{$\varepsilon$} -- 
++(-15:2) node[nodd,fill=yellow!60]{$\zeta$} -- 
 +(-90:1) node[leaf]{} +(0,0) --
++(15:2) node[nodd,fill=red!50]{$\eta$} --
 +(90:1) node[leaf]{} +(0,0) --
++(-15:2) node[nodd,fill=yellow!60]{$\theta$} --
 +(-90:1) node[leaf]{} +(0,0) --
++(15:2) node[nodd,fill=red!50]{$\iota$} --
 %+() +(0,0) --
++(80:1.8) node[nodd,fill=purple!70!blue!40]{$\kappa$} --
 +(170:1) node[leaf]{} +(0,0) --
 +(80:1) node[leaf]{} +(0,0) --
 +(-10:1) node[leaf]{} +(0,0)
++(80:-1.8) -- 
++(-15:2) node[nodd,fill=cyan!50]{$\lambda$}
 --
 +(-90:1) node[leaf]{} +(0,0) --
 %+(-15:1) node[leaf]{} +(0,0) --
 +(90:1) node[leaf]{} +(0,0)
 --
++(5:2) node[nodd,fill=brown!50]{$\mu$} --
 +(-105:1) node[leaf]{} +(0,0) --
 +(-30:1) node[leaf]{} +(0,0) --
 +(30:1) node[leaf]{} +(0,0) --
 +(105:1) node[leaf]{} +(0,0)
 (0,0) --
  +(45:1) node[leaf]{} 
 (0,0) --
 ++(120:1.7)  node[nodd,fill=blue!30]{$\gamma$} --
 +(40:1) node[leaf]{} +(0,0) --
 ++(150:1.7)  node[nodd,fill=red!50]{$\alpha$} --
 +(240:1) node[leaf]{} +(0,0) --
 +(60:1) node[leaf]{} +(0,0) --
 +(150:1) node[leaf]{} +(0,0)
 (0,0) --
 ++(190:2)  node[nodd,fill=green!40]{$\delta$} --
 +(-85:1) node[leaf]{} +(0,0) --
 ++(-180:2)  node[nodd,fill=red!50]{$\beta$} --
 %+(-270:1) node[leaf]{} +(0,0) --
 +(-140:1.1) node[leaf]{} +(0,0) --
 +(-220:1.1) node[leaf]{} +(0,0);
\end{tikzpicture}
$$
\caption{A decomposition tree}
\label{fig1}
\end{figure}
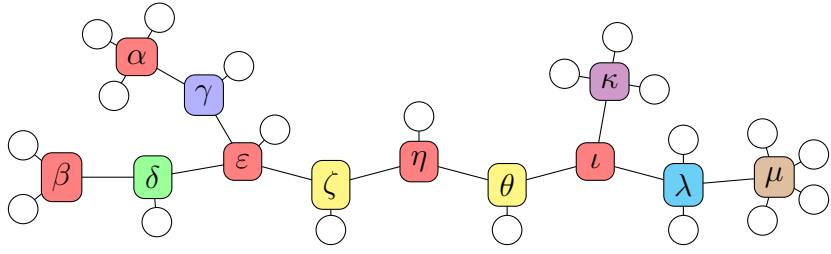

Since the number $V$ of nodes in a tree $T$ equals the number of vertices in the 
forest $G(T)$, the number $B$ of bridges in $T$ equals the number of edges in $G(T)$, 
and the number~$\mathit\Gamma$ of bunches in $T$ equals the number of connected 
components of $G(T)$, it follows that
\begin{equation}\label{eq:VGB}
\mathit\Gamma=V-B.
\end{equation}

It is evident that the number $\mathit\Gamma-L$ of non-bald bunches is less than 
or equal to the number $V-E$ of non-bald nodes. Therefore, the relations 
$\mathit\Gamma-L = V-B-L \leqslant V-E$ hold and from that we get a bound 
for the number of bald bunches
\begin{equation}\label{eq:L}
L \geqslant E-B.
\end{equation}

For two different leaves $x$ and $y$ in the leaf set of a bunch~$G$ in the graph~$G(T)$, 
we define an isotopy~$\psi^{x,y}$ of the decomposition $S$ in the following way. 
For any edge of the chain $P$ connecting leaves $x$ and $y$ in the tree~$T$ we take 
the involution $\xi=\xi(G)$ native to the nodes of the bunch $G$. Each bridge node 
in the chain $P$ does not belong to $G$, but is adjacent to two nodes of the bunch $G$ 
and one leaf $z$ of the tree $T$. If a bridge $v$ is labeled by~$f$, then for the leaf 
edge of $v$ we take a native transposition $\tau=\tau(z)$ of the bridge $v$ such 
that the three permutations $\xi$, $\xi$, $\tau$ in an appropriate order form an autotopy 
of the binary quasigroup~$f$. Such a transposition exists by Lemma~\ref{l:aut3}(iv). 
Finally, we take the identity permutation for the remaining edges of the tree~$T$.

\begin{lemma}\label{l:path}
For any two leaves $x$ and $y$ from the leaf set of a bunch in $G(T)$, 
the isotopy $\psi^{x,y}$ is an autotopy of the decomposition $S$.
\end{lemma}
\begin{proof}
Consider a bunch $G$ and any two leaves $x$ and $y$ from the leaf set 
of $G$. The nodes of the chain $P$ connecting $x$ and $y$ in the tree~$T$ can be partitioned 
into two parts. The first part consists of the nodes of the bunch $G$. If $u\in{P}$ is a node 
of $G$, then by construction the isotopy $\psi^{x,y}$ contains the involution $\xi=\xi(G)$ 
native to $u$ for each of the two edges incident to~$u$ in the chain~$P$ and identity 
permutations for all other edges incident to $u$ in the tree~$T$. 
By Lemma~\ref{l:2rodn},
such a collection of permutations forms an~autotopy 
of the quasigroup prescribed to the vertex $u$.

The second part of nodes in the chain~$P$ consists of bridges, 
which do not belong to the bunch $G$, 
but provide the edges of $G$. 
Let $v$ be such a bridge in the tree~$T$ and $z$ be the only leaf of $v$.
By construction, the isotopy $\psi^{x,y}$ 
contains two non-native to $v$ involutions $\xi(G)$ 
and a native to $v$ transposition 
$\tau(z)$ that form an~autotopy of the quasigroup prescribed to the vertex $v$.

For each of the nodes not in the chain $P$, 
the isotopy $\psi^{x,y}$ induces the identity autotopy. 
From these arguments, 
we conclude that for each node in the tree~$T$, 
the isotopy $\psi^{x,y}$ yields an autotopy 
of the quasigroup prescribed to this node. 
Consistently, 
$\psi^{x,y}$ is an autotopy of the decomposition $S$.
\end{proof}

Let us note that a bald bunch, 
as well as a bunch with only one leaf,
do not grant any autotopies of the kind $\psi^{x,y}$.

\begin{lemma}\label{l:V-B}
If a bunch $G$ contains $k\geqslant 1$ leaves in its leaf set, 
then there exist 
at least $2^{k-1}$ autotopies of the decomposition $S$ 
acting in the following way:
on the edges incident to leaves of $G$, 
they act with identity permutations or involutions 
native to the nodes of~$G$;
on the edges that are incident to the leaves of the 
bridges connecting the nodes of~$G$, 
they act with identity permutations 
or transpositions native to the bridges.
\end{lemma}

\begin{proof}
Let $\{x,y_1,\ldots,y_{k-1}\}$ be the leaf set of the bunch $G$. 
Autotopies $\psi^{x,y_i}$, 
$i=1,\ldots,k-1$, of order $2$ each commute with each other 
and are independent from each 
other since for each $i$ only one of them, 
namely $\psi^{x,y_i}$ obtains a non-identity 
permutation for the edge incident to the leaf $y_i$. 
Therefore, these $k-1$ autotopies 
yield $2^{k-1}$ autotopies corresponding to the bunch $G$.
\end{proof}

In further, 
the autotopies of the decomposition $S$ described in Lemma~\ref{l:V-B} 
are called the \emph{autotopies induced by the bunch $G$}. 
All them are of order $2$.

If a bunch contains a fork, we can point out autotopies of order $4$, 
which contribute additionally 
to the size of the autotopy group of the corresponding quasigroup.

\begin{lemma}\label{l:F}
For each fork in a decomposition tree, 
one can find two autotopies of the decomposition 
acting on the leaf edges of the fork 
with its native cycles and on all other edges 
with identity permutations.
\end{lemma}

\begin{proof}
Consider a fork $u$ in a decomposition $S$ 
and let $\xi$ be the native involution of~$u$. 
Without loss of generality, 
we assume that the node adjacent to the fork $u$ 
is its $0$-th neighbor, 
while the leaves $x$ and $y$ of $u$ are the $1$-st and $2$-nd neighbors respectively. 
For each pair $\varphi_1$, $\varphi_2$ of cycles native to $u$, by~Lemma~\ref{l:aut3} 
there exists exactly one permutation $\xi'\in\{\mathrm{Id},\xi\}$ such that the triple 
$\varphi=(\xi',\varphi_1,\varphi_2)$ forms an autotopy of the quasigroup $f$ prescribed 
to the fork $u$.

If $\xi'=\mathrm{Id}$, then $\varphi$ and $\varphi^{-1}$ are required autotopies of $f$,  
which can be finished~up to~autotopies of the decomposition $S$ by use of identity 
permutations. If $\xi'=\xi$, then one can take the isotopy 
$(\xi',\varphi_1,\varphi_2)(\xi,\xi,\mathrm{Id})=(\mathrm{Id},\varphi_1\xi,\varphi_2)$ 
instead of $\varphi$. By~Lemma~\ref{l:aut3}, the~former is also an autotopy of 
the quasigroup $f$ with its native cycles $\varphi_1\xi$ and $\varphi_2$.
\end{proof}

\begin{lemma}\label{l:independence}
Non-identity autotopies of a decomposition induced by different bunches 
are independent and commute with each other.
\end{lemma}

\begin{proof}
Consider a decomposition $S$ with the tree $T$ and two autotopies of $S$. 
If the supports 
of the autotopies intersect in the empty set, 
then they are trivially independent and commute 
with each other. 
At the same time, the supports of autotopies induced by different bunches 
of~$G(T)$ can intersect only in edges of bridge nodes. Indeed, 
according to Lemma~\ref{l:F}, 
the support of the autotopy corresponding to a fork 
does not exceed the set of leaf edges of the fork.
As while as the support of an autotopy~$\psi^{x,y}$ 
induced by a bunch $G$ can contain only edges incident to the nodes of $G$. 
If there are more than one node in~$G$, 
then some of the edges in the support of~$\psi^{x,y}$ are also 
inner edges of bridges connecting nodes of~$G$.

Thereby, it is sufficient to prove the lemma for autotopies induced by different bunches 
with supports intersecting in edges incident to bridge nodes.
An arbitrary autotopy induced by a bunch $G$ 
acts on edges in its support in the following way:
\begin{itemize}
 \item on edges incident to nodes of~$G$, it acts with involutions 
native to nodes of the bunch~$G$;
\item on leaf edges incident to bridges connecting nodes of~$G$,
it acts  with transpositions native to the bridges.
\end{itemize}
Assume that a bridge $v$ with a leaf $z$ connects two nodes of the bunch $G$;
and let $v$ be contained in another bunch $G'$. 
By~Lemma~\ref{l:aut3}(iv), all autotopies induced by~$G$ contain 
exactly one of the two transpositions native to $v$, namely $\tau=\tau(z)$. 
The transposition $\tau$ cannot generate the involution native to the bridge~$v$. 
Thus, the autotopies induced by the bunch~$G$ 
and the autotopies induced by the bunch~$G'$ 
are independent.

Consider autotopies $\theta=\psi^{x,y}$ and $\theta'=\psi^{x',y'}$ 
induced by the bunches~$G$ and~$G'$ respectively. 
Let the supports of these two autotopies contain edges incident 
to a bridge~$v$ in their intersection.
Autotopies~$\theta$ and~$\theta'$ act on inner 
edges of~$v$ with involutions $\xi=\xi(G)$ and~$\xi'=\xi'(G')$ 
native to $G$ and~$G'$ respectively. 
Since involutions $(01)(23)$, $(02)(31)$ and $(03)(12)$ commute with each other,
we have $\xi\xi'=\xi'\xi$.

The autotopy $\theta$ acts on the leaf edge of the bridge~$v$ with the transposition~$\tau$,
while the autotopy~$\theta'$ acts on this edge with some involution~$\eta'$
(or the identity permutation, which is considered trivially).
Both~$\tau$ and~$\eta'$ are native to~$v$.
It is obvious that, for example,
the involution $(01)(23)$ and the transpositions $(01)$ and $(23)$ forming it commute. 
The same is true for the involutions~$(02)(31)$ and~$(03)(12)$. Therefore, we get $\tau\eta'=\eta'\tau$.

From the reasoning above it is follows that the autotopies~$\theta$ and~$\theta'$ commute 
on every edge that is contained in the intersection of their supports. This proves 
commutativity of autotopies of the decomposition~$S$ induced by different bunches in its tree.
\end{proof}

\begin{theorem}\label{th:lb}
For an arbitrary $n$-ary quasigroup~$f$ of order~$4$, the following inequality holds:
\begin{equation}\label{eq:lb}
	|\mathrm{Atp}(f)|\geqslant2^{[n/2]+2}.
\end{equation}
If $n\geqslant5$, then this bound is sharp.
\end{theorem}

\begin{proof}
Let the quasigroup~$f$ have a reduced decomposition $S$ with the tree $T$. 
For each bunch~$G$ with $k\geqslant1$ leaves, 
by~Lemma~\ref{l:V-B} 
one can construct $2^{k-1}$ autotopies of~$f$ which act on variables corresponding 
to leaves of the bunch~$G$ with permutations of order~$2$. Taking into account all bunches 
of the graph~$G(T)$ except the bald ones, 
by use of Lemma~\ref{l:independence} we get 
$2^{N-(\mathit\Gamma-L)}$ autotopies of~$f$.

In addition, for any fork~$v$ in the tree $T$, by~Lemma~\ref{l:F} there are $2$ autotopies 
of~$f$ which act on~variables corresponding to the leaves adjacent to~$v$ with cycles native 
to the fork~$v$. This contributes the factor~$2^F$ to the number of constructed here 
autotopies of~$f$. In this way, using~(\ref{eq:VGB}) we obtain
\begin{equation}\label{eq:NVBLF}
|\mathrm{Atp}(f)|\geqslant2^{N-(\mathit\Gamma-L)+F}=2^{N-V+B+L+F}.
\end{equation}

Suppose that in the decomposition tree~$T$ there are $t$ edges and $V_s$ vertices 
of degree~$s$, $s=0,1,2,\ldots$. By definition of a quasigroup decomposition and accordingly 
to notation stabilized above, it can be written $V_1=N$, $V_2=0$. Thus,
$$N+\sum_{s\geqslant3}sV_s=2t=2(N+V-1).$$
It follows that 
\begin{equation}\label{eq:no-s5}
N+2V-2=\sum_{s\geqslant3}sV_s=4\sum_{s\geqslant3}V_s+\sum_{s\geqslant5}(s-4)V_i-V_3\geqslant4V-V_3.
\end{equation} 
Consequently, the inequality~$N\geqslant2V-V_3+2$ holds.

In accordance with the number of adjacent leaves, the nodes of degree~$3$ in the tree~$T$ 
are partitioned into forks, bridge nodes, and bald nodes (there are no vertices of degree~$3$ 
with three adjacent leaves since $n>2$). 
Trivially, the number of bald nodes 
of degree~$3$ is not greater 
than the total number of bald nodes in~$T$. 
Taking into account~(\ref{eq:L}), we get
\begin{equation}\label{eq:lys3} V_3\leqslant F+B+E\leqslant F+2B+L,\end{equation} 
which allows to rewrite the estimate for $N$ in more detail:
$$N\geqslant2V-V_3+2\geqslant2V-F-2B-L+2.$$
Hence,
$$-V+B\geqslant-\frac{1}{2}(N+F+L)+1.$$
Applying this inequality to~(\ref{eq:NVBLF}), we derive
\begin{equation}\label{eq:no-lys}
 |\mathrm{Atp}(f)|\geqslant 2^{(N+F+L)/2+1}\geqslant 2^{N/2+1}=2^{(n+3)/2}
\end{equation}
Let us note that the second inequality in~(\ref{eq:no-lys}) is strict if and only if 
the decomposition tree contains a fork or bald node. Since~$|\mathrm{Atp}(f)|$ is 
an integer and the number of autotopies generated by those described in Lemmas~\ref{l:V-B} 
and~\ref{l:F} is a power of $2$, %eventually it follows that
we have
\begin{equation}\label{eq:lowerest}
	|\mathrm{Atp}(f)|\geqslant2^{\lfloor n/2 \rfloor+2}.
\end{equation}

Further, let us show that the bound~(\ref{eq:lowerest}) is attainable. %for certain quasigroups. 
Consider the quasigroups~$l_3^\bullet$ and~$l_4^\bullet$ in Lemma~\ref{l:aut4}, which we denote 
here by~$f$ and~$h$ correspondingly, and the quasigroup 
$x_0=g(x_1,x_2,x_3)=\tau f (\tau x_1,\tau x_2,\tau x_3)$ with $\tau=(12)$, which is isotopic to 
the ternary quasigroup~$f$.

The quasigroups~$f$ and $g$ are $\{0,2\}$- and $\{0,1\}$-semilinear, respectively. 
Their autotopy groups are isomorphic to each other, namely one of them is conjugate 
with the other by the transposition~$\tau$. The permutation~$(01)(23)$ is a native involution 
for~$g$, as $(01)$ and~$(23)$ are native transpositions of~$g$.

Note that only the identity permutation~$\mathrm{Id}$ can be met in autotopies of both~$f$ and~$g$. 
The same is true for~$h$ and~$g$. Therefore, if~$f$ and~$g$ 
(or $h$ and $g$, respectively) are adjacent in a decomposition tree of some quasigroup~$q$, 
then by~Lemma~\ref{lem:composition} 
their autotopies can concatenate 
to an autotopy of the decomposition 
of~$q$ only by the identity permutation.

Reasoning in this way, it is easy to see that for odd $n\geqslant5$ the quasigroup~$q_n$ 
of arity~$n$ with a decomposition tree of kind
$$
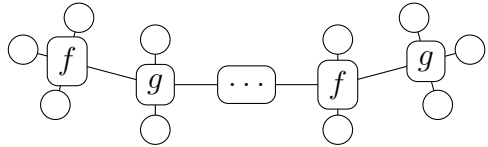
\begin{tikzpicture}
\draw (0,0) node[nodd]{$f$} --
 +(75:1) node[leaf]{} +(0,0) --
 +(-105:1) node[leaf]{} +(0,0) --
 +(165:1) node[leaf]{} +(0,0) --
++(-15:2) node[nodd]{$g$} --
 +(-90:1) node[leaf]{} +(0,0) -- 
 +(90:1) node[leaf]{} +(0,0) -- 
++(0:1) --
++(0:1) node[nodd]{$\ldots$} --
++(0:1) --
++(0:1) node[nodd]{$f$}  --
 +(-90:1) node[leaf]{} +(0,0) -- 
 +(90:1) node[leaf]{} +(0,0) -- 
++(15:2) node[nodd]{$g$} --
 +(-75:1) node[leaf]{} +(0,0) --
 +(15:1) node[leaf]{} +(0,0) --
 +(105:1) node[leaf]{} +(0,0);
\end{tikzpicture}
$$
has the autotopy group of order~$2^{(n-1)/2+2}$. 
For even $n\geqslant6$, the quasigroup~$q_n$ 
of arity~$n$ with a decomposition tree of kind
$$
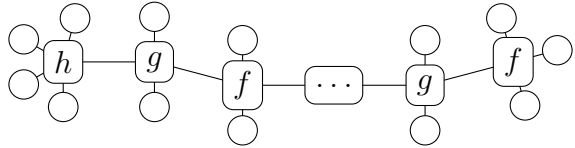
\begin{tikzpicture}
\draw (0,0) node[nodd]{$h$} --
 +(75:1) node[leaf]{} +(0,0) --
 +(-90:1) node[leaf]{} +(0,0) --
 +(150:1) node[leaf]{} +(0,0) --
 +(-150:1) node[leaf]{} +(0,0) --
++(0:2) node[nodd]{$g$} --
 +(-90:1) node[leaf]{} +(0,0) -- 
 +(90:1) node[leaf]{} +(0,0) -- 
++(-15:2) node[nodd]{$f$} --
 +(-90:1) node[leaf]{} +(0,0) -- 
 +(90:1) node[leaf]{} +(0,0) -- 
++(0:1) --
++(0:1) node[nodd]{$\ldots$} --
++(0:1) --
++(0:1) node[nodd]{$g$}  --
 +(-90:1) node[leaf]{} +(0,0) -- 
 +(90:1) node[leaf]{} +(0,0) -- 
++(15:2) node[nodd]{$f$} --
 +(-75:1) node[leaf]{} +(0,0) --
 +(15:1) node[leaf]{} +(0,0) --
 +(105:1) node[leaf]{} +(0,0);
\end{tikzpicture}
$$
has the autotopy group of order~$2^{n/2+2}$. 
In both cases, the equality is attained 
in~(\ref{eq:lowerest}) for the quasigroups designed.
\end{proof}

\begin{note} 
For $n=3$ and $n=4$, the bound pointed out in Theorem~\ref{th:lb} is not sharp. 
The quasigroup~$q_n$ described in the proof degenerates into a semilinear quasigroup, which has 
autotopies consisting of its native transpositions (see Lemma~\ref{l:aut4}(ii)). Such autotopies are not taken into account in the estimation of Theorem~\ref{th:lb}.

The quasigroup~$q_n$ also delivers 
the minimum for order of autotopy group in the case $n\in\{3,4\}$. 
However, in this case the minimum is two times greater than the minimum number in Theorem~\ref{th:lb}. 
At the same time,
any decomposition tree of a reducible quasigroup of arity~$n\in\{3,4\}$ contains 
a fork. If there are two forks, 
then the difference can be seen from inequalities in~(\ref{eq:no-lys}). 
If there is one fork, then by~Lemma~\ref{l:aut3}(iv) the quasigroup~$q_n$ 
has autotopies with 
non-identity involution acting on the inner edge of the fork,
which are not considered in the proof of Theorem~\ref{th:lb}.
\end{note}

%=================================
%=================================

\subsection{Quasigroups with autotopy groups of minimum order}
\label{s:min}

Besides the examples of quasigroups described in the proof of Theorem~\ref{th:lb}, there are 
many other such quasigroups for which the equality is attained in the lower bound given 
by the theorem. In this section, we characterize quasigroups with this property for $n$ odd. 
In our reasoning we examine the cases in which all non-strict inequalities occurring in the proof 
of the theorem turn into equalities. In case of even $n$, we have not got such an opportunity 
since we explicitly use the ceil function.

For odd $n$, the number in the right part of~\eqref{eq:lb} equals 
$2^{\lfloor n/2 \rfloor +2}=2^{\frac{n-1}{2}+2}=2^{\frac{n+3}{2}}$. 
Based on the proof of Theorem~\ref{th:lb}, for odd $n$ one can derive 
properties of reduced decompositions with exactly $2^{\frac{n+3}{2}}$ autotopies. 
The tree of such a decomposition does not contain:
\begin{asparaenum}[\hfill(I)]
    \item any vertices of degree greater than $4$ 
          (this follows from the equality in~(\ref{eq:no-s5})),
    \item any forks (the equality in~(\ref{eq:no-lys})),
    \item any bald bunches (the equality in~(\ref{eq:no-lys})), 
    \item greater than one non-bald vertex in each bunch (the equality in~(\ref{eq:L})),
    \item any bald vertices of degree greater than~$3$ (the equality in~(\ref{eq:lys3})).
\end{asparaenum}
Conditions (III)--(V) imply that in a decomposition with exactly $2^{\frac{n+3}{2}}$ autotopies 
\begin{asparaenum}
    \item[(III--V)] each bunch contains exactly one non-bald node, which can be a bridge or node 
	of degree~$4$, and possibly some bald nodes of degree~$3$ as well.
\end{asparaenum}

Let us take a reduced decomposition that satisfies the conditions~(I)--(V) and label each of its nodes 
of degree~$4$ with a ternary quasigroup isotopic to~$l_3^\bullet$. 
This  decomposition has only the autotopies considered in the proof of Theorem~\ref{th:lb},
and the corresponding quasigroup meets the bound $2^{\frac{n+3}{2}}$ for the order of the autotopy group. 
On the other side, the following statement takes place.

\begin{lemma}\label{l:no-z4z4-z2z4}
Let $S$ be a reduced decomposition of an $n$-ary quasigroup $f$ of order $4$ with 
$|\mathrm{Atp}(f)|=2^{\frac{n+3}{2}}$. Each node of degree~$4$ in the decomposition $S$ has 
a ternary quasigroup isotopic to the quasigroup $l_3^\bullet$ as its label.
\end{lemma}
\begin{proof}
There exist exactly five ternary quasigroups up to isotopy, variable permutation, 
and inversion~\cite{ZZ:2002}. One of them is linear and another one is non-semilinear. 
The remaining three quasigroups are semilinear, but not linear. 
These three are the quasigroups $l_3^\bullet$,
$$
	g(x_1,x_2,x_3) = x_1\oplus (x_2 +_4 x_3),\qquad \mbox{and}\quad h(x_1,x_2,x_3) = x_1 +_4 x_2 +_4 x_3.
$$
The quasigroup~$g$ admits the autotopy $\varphi=((01)(23),(01)(23),\mathrm{Id},\mathrm{Id})$ 
with two involutions non-native to~$g$. The quasigroup~$h$ has got the autotopy 
$((01)(23),(01)(23),(01)(23),(01)(23))$ consisting of four non-native involutions. 

Suppose that the tree~$T$ of a decomposition~$S$ has a node $\alpha$ labeled with~$g$
(the case of~$h$ can be handled similarly). 
%Given the quasigroup~$g$ in the decomposition~$S$, 
We will show that under the assumptions 
made one can obtain $|\mathrm{Atp}(f)|>2^{\frac{n+3}{2}}$. 
With this aim, for the decomposition~$S$ 
we construct a special autotopy consisting of permutations in the transformation~$\varphi$. 
It~helps us to get the inequality. 

Consider the $0$-th neighbor of the node~$\alpha$. 
If it is a leaf, then we prescribe 
the permutation~$(01)(23)$ to the leaf edge. 
If the neighbor is a node, 
it belongs to some bunch~$G$. 
According to~(III), the bunch~$G$ is not bald and contains at least one leaf~$x$.
Assume that a chain~$P$ connects the leaf~$x$ with the node~$\alpha$ in the tree~$T$. Prescribe 
the permutation~$(01)(23)$ to each edge of~$P$. If there is any bridge node in the chain~$P$, 
which connects two nodes of the bunch~$G$, we prescribe a native transposition to the leaf edge of 
the bridge accordingly to Lemma~\ref{l:aut3}(iv). Next, we do the same construction for the $1$-st 
neighbor of the node~$\alpha$ and prescribe the identity permutation to the remaining edges 
in the tree $T$. 

Finally, we obtain an autotopy~$\theta$ of the decomposition~$S$, because for each node~$v$ in~$T$ 
the permutations acting on the edges incident to~$v$ form an autotopy of the quasigroup 
corresponding to $v$. In addition, the autotopy~$\theta$ does not contribute to the number 
given as a lower bound in Theorem~\ref{th:lb}. Indeed, in the proof of the theorem we consider 
only those autotopies of the decomposition which act on every edge incident to a node of degree~$4$ 
with an involution native to the node. In contrast, the autotopy~$\theta$ acts on two edges 
incident to~$\alpha$ with the involution~$(01)(23)$, which is not native to~$g$. Consequently, 
$\theta$ increases the order of the autotopy group of~$f$; so, 
$|\mathrm{Atp}(f)|>2^{\frac{n+3}{2}}$.

If the quasigroups~$h$ is prescribed to the node~$\alpha$, 
then one can construct an additional autotopy 
of~$f$ in the same way. 
The only difference is that in this case all neighbors 
of the node~$\alpha$ should be considered.
\end{proof}

A decomposition tree satisfying conditions~(I)--(V) 
can be constructed using the following procedure.

\medskip\noindent
\textbf{Construction T}\par
\smallskip
\begin{asparaenum}[{\textsc{Step}}~1.]\itemsep=\smallskipamount
	\item Take an arbitrary tree $T_1$ with exactly $(n-1)/2$ vertices, which we call 
	\emph{nodes}. The degree of each node should not exceed $3$.
	\item Connect $4-i$ new \emph{leaves} to each node of degree~$i\in\{1,2,3\}$ in the tree~$T_1$. 
	Degree of each node in the resulting tree~$T_2$ equals $4$.
	\item Select some (maybe none, maybe all) nodes adjacent to exactly one leaf. % each in~$T_2$. 
	Replace each selected node~$s$ by two new nodes~$u_s$ and~$v_s$ of degree~$3$ adjacent 
	to each other. Four neighbors of~$v$ can be distributed among the neighborhoods of~$u_s$ and~$v_s$ 
	in any of three possible ways. Denote the tree obtained at this step by~$T_3$.
	\item Divide the nodes of~$T_3$ into two independent parts~$V_1$ and~$V_2$, which is possible because 
	any tree is a bipartite graph.
	\item To each node in the part~$V_i$, $i=1,2$,
	assign a $\{0,i\}$-semilinear quasigroup isotopic to~$+_4$ or~$l_3^\bullet$ % as its label correspondingly to the degree of the node.
	\item Finally, choose a leaf to represent the value of the quasigroup, index the neighbors 
	of each node in an appropriate way and get a decomposition tree $T$ of some quasigroup~$f$.
\end{asparaenum}

\medskip
Moreover, it turns out that, for every quasigroup~$f$ which meets the bound on the order of its autotopy group, one can build a quasigroup isotopic to~$f$ using Construction~T.

\begin{theorem}\label{th:min-all}
Every $n$-ary quasigroup~$f$ with the autotopy group of order~$2^{\frac{n+3}2}$ is isotopic 
to some quasigroup with a decomposition tree obtained with Construction T.
\end{theorem}
\begin{proof}
Let~$T$ be a decomposition tree built using the construction. Nodes of degree~$3$ are combined 
in pairs ``bald node -- bridge node'' input at Step 3. 
In each pair, the bald node is included 
into some bunch, while the bridge node corresponds to an edge of that bunch. 
Since a bunch is 
a tree and the number of vertices in a tree is one more than the number of edges, 
every bunch contains exactly one non-bald node. 
By construction, it is a bridge node or node of degree~$4$. 
Therefore, the tree~$T$ satisfies conditions~(I)--(V).

On the other hand, consider a quasigroups~$f$ that meets the bound~$2^{\frac{n+3}2}$ 
on the order of the autotopy group. 
Let~$T$ be the tree of a decomposition of~$f$. 
From condition~(III--V), it follows that in any bunch 
the numbers of bald nodes and edges coincide. 
Consequently, there is a one-to-one correspondence between the bridges and the bald nodes, 
all of which have degree~$3$. 
In addition, we can require that a bridge is adjacent to its corresponding bald node. 
Shrinking the pairs of corresponding nodes of degree~$3$ (an operation 
reverse to Step~3), we get a tree whose all nodes are non-bald and of degree~$4$. 
Any such a tree can be obtained at Step~2. By Lemmas~\ref{l:isotop} and~\ref{l:no-z4z4-z2z4}, 
the node labeling of the decomposition tree~$T$ conforms to the labeling at Step~5.
\end{proof}
 
%=================================
%=================================
%=================================
\section{An upper bound}\label{s:up}

In this section, we prove that the maximum order of the autotopy group of an $n$-ary quasigroup 
of order~$4$ equals $6\cdot 2^n$, 
and only the linear quasigroup, 
which is unique up to isotopy, 
reaches the upper bound. 
We also determine the quasigroups that have the maximum order 
of autotopy groups among the nonlinear quasigroups 
and point out this order as well.

Here we use Orbit--Stabilizer Theorem. 
In our case, this theorem says that 
the order of the autotopy group of a code $M$
equals the size of the stabilizer of any element $x\in M$ 
multiplied by size of the orbit of~$x$ under the action 
of the autotopy group. 
Let us start with several auxiliary statements concerning autotopies 
of a $n$-ary quasigroups that stabilize a certain element in the quasigroup code. For simplicity,  that element  is usually considered 
to be the all-zero tuple $(0,\ldots,0)$. 
The next lemma takes place 
for a quasigroup of any order~$k$.

\begin{lemma}
\label{lem:symmetry}
Let~$f$ be an $n$-ary quasigroup and $f(0,\ldots,0)=0$. 
Then an arbitrary autotopy 
$(\theta_0,\ldots,\theta_n)\in\mathrm{Atp}(f)$ stabilizing 
the all-zero tuple is uniquely determined 
by any single of its permutations $\theta_i$, $i\in\{0,\ldots,n\}$. 
In particular, if for some 
$i\in\{0,\ldots,n\}$ the permutation $\theta_i$ is identity, then all the others are also identity.
\end{lemma}
\begin{proof}
Without loss of generality, given the permutation $\theta_0$, we express the permutations 
$\theta_1,\ldots,\theta_n$ in terms of it. 

Assume that $(\theta_0,\ldots,\theta_n)$ is an autotopy of $f$ such that
$\theta_i0=0$, $i=0,1,\ldots,n$. By the 

By the autotopy definition, we get 
$$\theta_0^{-1} f(\theta_1 x_1,0,\ldots,0) = f(x_1,0,\ldots,0)\quad\text{for any $x_1\in\Sigma,$}$$
which is equivalent to
$$\theta_1 x_1 = f^{<1>}(\theta_0 f(x_1,0,\ldots,0),0,\ldots,0)\quad\text{for any $x_1\in\Sigma$}.$$
One can see that the permutation $\theta_1$ is entirely determined by the quasigroup~$f$ and permutation~$\theta_0$. 
In the same manner, we can express any of $\theta_0$, \ldots, $\theta_n$
through any other one.

Finally, by the argumentation above it is evident that, for example, $\theta_0=\mathrm{Id}$ imply 
$\theta_i=\mathrm{Id}$ for any $i\in\{1,\ldots ,n\}$.
\end{proof}

\begin{corollary}
\label{c:order}
Let $f$ be an $n$-ary quasigroup with an autotopy
$\theta=(\theta_0,\ldots,\theta_n)$.
If $\theta$ stabilizes a tuple $(a_0,\ldots,a_n)$ from ${M(f)}$, 
then all of the permutations $\theta_i$, 
$i=0,\ldots,n$, have the same order.
\end{corollary}
\begin{proof}
By the hypothesis, $(a_0,\ldots,a_n)\in{M(f)}$; 
that is $a_0=f(a_1,\ldots,a_n)$. 
Consider the quasigroup $g$ defined as 
$$g(x_1,\ldots,x_n)=\tau_0f(\tau_1x_1,\ldots,\tau_1x_n),$$ 
where 
$\tau=(\tau_0,\ldots,\tau_n)$ is an isotopy consisting of the transpositions $\tau_i=(0\,a_i)$, 
$i=0,\ldots,n$. 
It is easy to verify that $g(0,\ldots,0)=0$ 
and the isotopy 
$$\delta=\tau\theta\tau=(\tau_0\theta_0\tau_0,\ldots,\tau_n\theta_n\tau_n),$$ conjugate to $\theta$, 
is an autotopy of $g$ stabilizing the all-zero tuple.

By Lemma~\ref{lem:symmetry}, for any integer $r$, 
all permutations from the autotopy $\delta^r$ 
are identity if there is at least one identity permutation among them. Consequently, all 
permutations~$\delta_i$, $i=0,\ldots,n$, have the same order.

It remains to note that the permutations $\delta_i$, $\theta_i$ are
of the same order
because $\delta_i^r=\tau\theta_i^r\tau$, 
$i=0,\ldots,n$.
\end{proof}

\begin{lemma}
\label{l:2-to-semi}
Let $f$ be an $n$-ary quasigroup of order~$4$ such that $f(0,\ldots,0)=0$. 
If $f$ has an autotopy~$\theta$ of order~$2$ 
that stabilizes the all-zero tuple, then~$f$ is semilinear.
\end{lemma}
\begin{proof}
By Corollary~\ref{c:order}, each of the permutations $\theta_i$, $i=0,\ldots,n$, has order $2$. 
Since $\theta_i(0)=0$ for each $i=0,\ldots,n$, we have $\theta_i\in \{(12),(13),(23)\}$. 
Without loss of generality,
assume that $\theta_0=\ldots=\theta_n=(23)$
(otherwise, consider 
a quasigroup isotopic to~$f$ 
that has the autotopy consisting of permutations $(23)$).

For every $(x_1,\ldots,x_n)$ from $\{0,1\}^n$ and for $x_0=f(x_1,\ldots,x_n)$, we have
$$
	x_0=f(x_1,\ldots,x_n) = \theta_0^{-1}f(\theta_1x_1,\ldots,\theta_nx_n) 
	= \theta_0f(x_1,\ldots,x_n) = \theta_0x_0.
$$
Since $\theta_0=(23)$, the value of $x_0$ can only be $0$ or $1$.
Therefore, the quasigroup~$f$ maps 
$\{0,1\}^n$ to $\{0,1\}$. 
So, $f$ is semilinear by the definition.
\end{proof}

\begin{lemma}
\label{l:3-to-lin}
Let $f$ be an $n$-ary quasigroup of order~$4$ 
such that $f(0,\ldots,0)=0$. 
If $f$ has 
an autotopy~$\theta$ of order~$3$ 
that stabilizes some tuple 
$(a_0,\ldots,a_n)\in M(f)$, 
then~$f$ is linear.
\end{lemma}

\begin{proof}
By Corollary~\ref{c:order},
each of the permutations 
$\theta_i$, $i=0,\ldots,n$, 
has order $3$. 

(i)
If $f$ is $\{0,1\}$-, $\{0,2\}$-, or $\{0,3\}$-semilinear 
in every variable, 
then it is $\{a,b\}$-semilinear 
for any $a\neq b\in\Sigma$ and, consequently, linear.
Hence, the lemma is true for semilinear quasigroups.

(ii)
Assume $f$ is not semilinear. 
Consider a proper decomposition~$S$ 
of~$f$ and the corresponding tree~$T$. 
The autotopy~$\delta$ of~$S$ induced by~$\theta$ 
has the same order~$3$. 
Therefore, $\delta$ consists of $3$-cycles or identity permutations. Each of those permutations stabilizes some 
element in $\Sigma$.

Consider an arbitrary non-bald node~$v$
labeled by a quasigroup~$g$.
The autotopy of~$g$ 
induced by~$\delta$ satisfies the hypothesis of 
Corollary~\ref{c:order};
so, each of its permutations has order $3$. 
Since~$S$ is a proper 
decomposition, 
$g$ is semilinear,
and from item (i) of this proof 
we conclude that it is linear. 
This contradicts the definition of a proper decomposition.
\end{proof}

\begin{theorem}
(i)
The maximum order for an autotopy group of an $n$-ary quasigroup of order~$4$ equals $6\cdot4^n$;
 only the linear quasigroups reach this maximum.
(ii)
The maximum order for an autotopy group 
of a nonlinear $n$-ary quasigroup of order~$4$ 
equals $2\cdot4^n$;
 only the semilinear quasigroups 
 whose autotopy group acts transitively 
 on their codes reach this maximum. 
\end{theorem}
\begin{proof}
Consider an arbitrary $n$-ary quasigroup~$f$ of order~$4$. Without loss of generality, 
we assume that~$f(0,\ldots,0)=0$. 
% Let us note that one can easily obtain the equality 
% $\mathrm{Atp}(f)=\mathrm{Atp}(M(f))$ using definitions.

By Orbit--Stabilizer Theorem, the order of~$\mathrm{Atp}(M(f))$ 
equals the size of its stabilizer subgroup 
with respect to $(0,\ldots,0)\in{M(f)}$ multiplied by 
the size of the orbit of $(0,\ldots,0)$ 
under the action of~$\mathrm{Atp}(M(f))$.
 
For the all-zero tuple, 
the size of its orbit 
does not exceed the cardinality of~$M(f)$, i.e., $4^n$
(the equality takes place if and only if 
the orbit coincides 
with the code;
in other words if 
the action of the autotopy group is transitive on the code.)
 
Next, consider the size of the stabilizer with respect to the all-zero tuple. 
For a non-semilinear quasigroup, it equals~$1$ by Lemmas~\ref{l:2-to-semi} and~\ref{l:3-to-lin}. 
As for a semilinear quasigroup that is not linear, the size of the stabilizer is~$2$ (at least~$2$ by Lemma~\ref{l:aut4};
at most~$2$ by Lemmas~\ref{lem:symmetry} and~\ref{l:3-to-lin}). 
So, (ii) is proved.

Since any linear quasigroup is isotopic to the quasigroup 
$l_n(x_1,\ldots,x_n) = x_1\oplus\ldots\oplus x_n$, 
it remains to find $|\mathrm{Atp}(l_n)|$. 
For an arbitrary tuple~$(a_0,\ldots,a_n)\in{M(l_n)}$, 
the mapping $(x_0,\ldots,x_n)\mapsto(x_0\oplus a_0,\ldots,x_n\oplus a_n)$
maps $(0,\ldots,0)$  to $(a_0,\ldots,a_n)$ and induces an autotopy of~$l_n$. 
Hence, the size of the orbit of~$(0,\ldots,0)$ equals~$4^n$.

The size of the stabilizer with respect to $(0,\ldots,0)$ 
is at most~$3!$ 
by Lemma~\ref{lem:symmetry}. 
On the other hand, 
for each of $3!$ permutations $\theta_*$ of $\Sigma$ 
such that $\theta_*(0)=0$,
we have ab autotopy~$\theta=(\theta_*,\ldots,\theta_*)$ 
(this can be checked by induction on the arity~$n$). 
Therefore, the size 
of the stabilizer equals~$6$, 
and the order of the autotopy group 
of any linear $n$-ary quasigroup 
of order~$4$ is~$6\cdot4^n$.
\end{proof}

In conclusion, we should note that the semilinear $n$-ary quasigroups 
with transitive autotopy groups were characterized in~\cite{KP:2016:transitive}, where a correspondence 
between such quasigroups and Boolean polynomials of degree at most $2$ 
was established.

% \bibliographystyle{plain}
% \bibliography{../../k.bib}
% 
% \end{document}

\providecommand\href[2]{#2} \providecommand\url[1]{\href{#1}{#1}}
  \def\DOI#1{{\small {DOI}:
  \href{http://dx.doi.org/#1}{#1}}}\def\DOIURL#1#2{{\small{DOI}:
  \href{http://dx.doi.org/#2}{#1}}}

\end{document}

\maketitle
\begin{abstract}
    Алгебраическая система из конечного множества $\Sigma$ мощности $k$ и $n$-местной операции $f$, обратимой в каждом аргументе, называется 
    $n$-арной квазигруппой порядка $k$.
    Автотопией $n$-арной квазигруппы $(\Sigma,f)$ называется такой
    набор $(\theta_0,\theta_1,\ldots,\theta_n)$ из $n+1$ перестановки множества $\Sigma$, что 
    $f(\theta_1(x_1),\ldots,\theta_n(x_n))\equiv \theta_0(f(x_1,\ldots,x_n))$. 
    Показано, что любая $n$-арная квазигруппа порядка $4$ имеет не меньше чем $2^{[n/2]+2}$ и не больше чем $6\cdot 4^n$ автотопий.
    Получена характеризация $n$-арных квазигрупп порядка $4$ с 
    $2^{(n+3)/2}$, $2\cdot 4^n$, и $6\cdot 4^n$ автотопиями.
%An algebraic system from a finite set $\Sigma$ of cardinality $k$ and an $n$-ary operation $f$ that is invertible in each argument is called an $n$-ary quasigroup of order $k$. An autotopy of an $n$-ary quasigroup $(\Sigma,f)$ is a set $(\theta_0,\theta_1,\ ldots,\theta_n)$ from $n+1$ permutations $\Sigma$ such that $f(\theta_1(x_1), ...,\theta_n(x_n))=\theta_0(f(x_1, ..., x_n))$. We show that any $n$-ary quasigroup of order $4$ has at least $2^{[n/2]+2}$ and no more than $6\cdot 4^n$  autotopies. We characterize the $n$-ary quasigroups of order $4$ with $2^{(n+3)/2}$, $2\cdot 4^n$, and $6\cdot 4^n$ autotopies.

Ключевые слова: полиадическая квазигруппа, $n$-арная квазигруппа, группа автотопий.
%Keywords: polyadic quasigroup, $n$-ary quasigroup, autotopy group

MSC: 20N05, 08A35
\end{abstract}

%=================================
%=================================
%=================================
\section{Введение}\label{s:intro}

Обозначим через $\Sigma$ множество из $k$ элементов $0$, $1$, \dots, $k-1$. 
Декартова степень~$\Sigma^n$ 
состоит из всех упорядоченных наборов длины $n$, образованных элементами из $\Sigma$. 
Алгебраическая система с носителем $\Sigma$ и $n$-арной операцией $f\colon\Sigma^n\to\Sigma$, 
обратимой по каждому аргументу, называется \emph{$n$-арной квазигруппой порядка $k$} (иногда, для краткости, 
$n$"=квазигруппой или просто квазигруппой). 
Соответствующую операцию $f$ также принято называть квазигруппой.
% График $\{(x_0,x_1,\ldots,x_n) \mid x_0=f(x_1,\ldots,x_n)\}$ квазигруппы будем называть ее \emph{кодом}
% (в теории кодов, исправляющих ошибки, такие множества известны как МДР коды с расстоянием $2$).

\emph{Изотопией} множества $\Sigma^{n+1}$ называется 
набор $\theta=(\theta_0,\theta_1,\ldots,\theta_n)$ перестановок из симметрической группы~$S_k$, 
действующей на $\Sigma$. 
Действие изотопии на  $\Sigma^{n+1}$ 
задается по правилу
$$\theta\colon{x}\mapsto\theta({x})=(\theta_0(x_0),\ldots,\theta_n(x_n))\quad\mbox{для } x=(x_0,\ldots,x_n)\in\Sigma^{n+1}.$$
Для обозначения изотопий и состовляющих их перестановок мы будем использовать греческий алфавит, 
а при записи их действия на элементы $\Sigma$ иногда будем опускать скобки.

Множества $M_1,M_2\subseteq\Sigma^{n+1}$ называются \emph{изотопными}, если 
существует изотопия $\theta$ такая, что $\theta(M_1)=M_2$. Квазигруппы $f$ и $g$ 
называются \emph{изотопными}, если 
% изотопны соответствующие им коды $M(f)$ и $M(g)$,
% то есть 
для некоторой изотопии $\theta=(\theta_0,\theta_1,\ldots,\theta_n)$ выполнено
\begin{equation}\label{eq:q-iso}  g(x_1,,\ldots,x_n)=\theta_0^{-1} f(\theta_1 x_1,\ldots,\theta_n x_n). \end{equation}
% Заметим, что при этом коды $M(f)$ и $M(g)$ связаны обратной изотопией, а именно $M(g)=\theta^{-1}M(f)$.
% Изотопия $\theta=(\theta_0,\theta_1,\ldots,\theta_n)$ называется \emph{главной}, если $\theta_0$ --- тождественная перестановка.
Если же $f=g$, то любая изотопия $\theta$, для которой выполнено (\ref{eq:q-iso}), называется \emph{автотопией} квазигруппы $f$.
\textit{Группой автотопий}  $\mathrm{Atp}(f)$ квазигруппы $f$ называется группа, % $\mathrm{Atp}(f)=\mathrm{Atp}(M(f))$, 
состоящая из всех автотопий квазигруппы $f$ (групповой операцией является суперпозиция).
% то есть изотопиями $\Sigma^{n+1}$, отображающими код $M(f)$ на себя.

Напомним, что $2$"=квазигруппа $f$ с нейтральным элементом $e$, для которого при любом $a\in\Sigma$ 
имеют место равенства $f(e,{a})=f(a,{e})=a$, является \emph{лупой}. Если же лупа $f$
удовлетворяет аксиоме ассоциативности $f(x,f(y,z)) \equiv f(f(x,y),z)$, то перед нами уже группа. 
Известно (см., например,~\cite{RU:Belousov.ru}), 
что все $2$"=квазигруппы порядка $4$ изотопны либо группе $Z_2\times{Z_2}$, либо группе $Z_4$. 
Таким образом, квазигруппа может иметь групповые характеристики, которые иллюстрируют 
ее алгебраическую природу.

Вместе с тем, понятие квазигруппы допускает чисто комбинаторное толкование. Назовем \textit{линией} 
в $\Sigma^{n+1}$ подмножество из $n$ элементов, которые попарно отличаются ровно в одной координате. 
Для квазигруппы $f\colon\Sigma^n\to\Sigma$ множество 
$M(f)=\{(x_0,x_1,\ldots,x_n)\in\Sigma^{n+1}\mid x_0=f({x})\}$ будем называть \textit{кодом} 
квазигруппы $f$. Термин <<код>> заимствован из теории кодов, исправляющих ошибки, 
в~рамках которой множество $M(f)$ является МДР-кодом с расстоянием~$2$ 
(эквивалентным понятием, также хорошо известным в комбинаторике, является латинский гиперкуб). 
Код квазигруппы характеризуется как подмножество $\Sigma^{n+1}$ мощности~$k^n$, пересекающее 
каждую линию ровно по одному элементу. Такой взгляд позволяет увидеть квазигруппу 
с ее комбинаторной стороны.
Заметим, что коды изотопных квазигрупп изотопны,
а именно, из (\ref{eq:q-iso}) следует, что $M(g)=\theta^{-1}(M(f))$.

В настоящей статье исследуются автотопии квазигрупп порядка $4$. Установлены точные верхняя и нижняя 
оценки на порядок группы автотопий такой квазигруппы. В некотором роде является естественным тот факт, 
что наиболее богатой группа автотопий оказалась у квазигрупп, называемых линейными, со~структурой, 
близкой к групповой. 
Также в статье характеризуются квазигруппы с минимальным и предмаксимальным (то есть следующим за максимальным) 
порядками группы автотопий.

Понятие автотопии представляет собой обобщение более частного понятия <<автоморфизм>> и отражает в некотором 
смысле <<регулярность>> или <<симметричность>> квазигруппы как комбинаторного объекта.
Изучение преобразований объемлющего 
пространства, отображающих объект на себя, является классической, но вместе с тем и трудной задачей, 
рассматриваемой во многих областях математики. Сложность подобных задач иллюстрируется теоремой 
Фрухта~\cite{RU:Frucht:38} о том, что каждая конечная группа изоморфна группе автоморфизмов некоторого 
графа, а также аналогичным результатом, касающимся совершенных кодов в теории кодирования~\cite{RU:Phelps:86}.

В теории кодирования изучается группа автоморфизмов кода, порожденная изометриями метрического 
пространства, относительно которых этот код инвариантен. Здесь мы находим еще один пример того, 
что объект с групповыми свойствами имеет наиболее богатую группу автоморфизмов. 
В работах~\cite{RU:Mal:2000:aut.ru}, \cite{RU:SolTop2000ppi.ru}, \cite{RU:SolTop2000} показано, что двоичный код Хэмминга, линейный 
совершенный код, имеет максимальную по мощности группу автоморфизмов среди двоичных совершенных кодов, 
причем ее порядок минимум вдвое превосходит порядки групп автоморфизмов других двоичных совершенных 
кодов той же длины.

Предлагаемая статья организована следующим образом. В разделе~\ref{s:def} приведены базовые определения и утверждения. 
В разделе~\ref{s:repres} раскрывается представление квазигрупп, необходимое для дальнейшего доказательства основных 
результатов. Вспомогательные утверждения об автотопиях квазигрупп собраны в разделе~\ref{s:au}. Доказательство 
нижней оценки на порядок группы автотопий квазигруппы порядка $4$ следует в разделе~\ref{s:lb}. В разделе~\ref{s:min} 
обсуждаются квазигруппы, имеющие наименьшую по мощности группу автотопий, порядок которой достигает 
нижней границы. Наконец, в разделе~\ref{s:up} устанавливается верхняя оценка на порядок группы автотопий 
квазигруппы порядка $4$ и доказывается, что она достигается только линейными квазигруппами.
Там же установлен максимальный порядок группы автотопий нелинейной квазигруппы порядка $4$ и доказано,
что он достигается только на изотопно транзитивных квазигруппах, описанных в работе~\cite{RU:KP:2016:transitive}.

%=================================
%=================================
%=================================
\section{Обозначения и основные факты}\label{s:def}

% \hspace*{\parindent} Пусть $\Sigma=\{0,1,\dots,k-1\}$ и $[m..n]=\{m,m+1,\ldots,n\}$ [??? продумать удобное обозначение, объединять каждый раз с 0 неудобно. И место явно не здесь]. Декартова степень 
% $\Sigma^n$ состоит из всех упорядоченных наборов длины $n$, образованных элементами из $\Sigma$. 
% Алгебраическая система с носителем $\Sigma$ и $n$-арной операцией $f\colon\Sigma^n\to\Sigma$, 
% обратимой по каждой переменной, называется \emph{$n$-арной квазигруппой порядка $k$} (иногда, для краткости, 
% $n$"=квазигруппой или просто квазигруппой). Функцию $f$ также принято называть квазигруппой.
% Ассоциативная $2$"=квазигруппа с нейтральным элементом является группой.

% Пусть $[m..n]=\{m,m+1,\ldots,n\}$ [??? продумать удобное обозначение, объединять каждый раз с 0 неудобно. И место явно не здесь]. 

Для ${x}=(x_1,\ldots,x_n)\in\Sigma^n$ и $a\in\Sigma$ положим ${x}_i^a=(x_1,\ldots,x_{i-1},a,x_{i+1},\ldots,x_n)$. 
\emph{Обращение} $n$"=квазигруппы $f$ 
по $i$-му аргументу обозначим через $f^{\langle{i}\rangle}$, то есть для любых 
${x}\in\Sigma^n$ и $a\in\Sigma$ равенство $f^{\langle{i}\rangle}({x}_i^a)=x_i$ эквивалентно 
равенству $f({x})=a$. Очевидно, обращение $n$"=квазигруппы по произвольному аргументу является 
$n$"=квази\-группой. Под $0$-м аргументом $n$"=квазигруппы $f$ будем понимать значение функции $f(x_1,\ldots,x_n)$, 
которое, формально не являясь аргументом самой операции $f$, ассоциируется с $i$-м аргументом её обращения $f^{\langle{i}\rangle}$.

В настоящей статье исследуются автотопии квазигрупп порядка $4$, то есть всюду далее 
$\Sigma=\{0,1,2,3\}$. 
% Известно (см., например,~\cite{RU:Belousov.ru}), что все $2$"=квазигруппы порядка $4$ 
% изотопны либо группе $Z_2\times{Z_2}$, либо группе $Z_4$. 
% 
% Все $2$"=квазигруппы, изотопные $Z_2\times{Z_2}$, 
% отображаются в эту группу некоторой главной изотопией. Здесь и далее будем полагать, что $0\in\Sigma$ 
% выступает в качестве нейтрального элемента. 
% $2$"=квазигруппы, изотопные группе $Z_4$, 
% относительно главной изотопии образуют три класса эквивалентности. Класс определяется тем, какой 
% элемент группы --- $1$, $2$ или $3$ --- имеет порядок два в групповом представителе данного класса.
% 
Квазигруппа $f$ порядка $4$ называется \emph{полулинейной}, если найдутся 
$a_j,b_j\in\Sigma$, $a_j\ne b_j$, $j=0,1,\ldots,n$, для которых имеет место
\begin{equation}\label{eq:semilinear}
	f(\{a_1,b_1\}\times\ldots\times\{a_n,b_n\}) = \{a_0,b_0\}.
\end{equation}
В этом случае будем также говорить, что квазигруппа $f$ является \textit{$\{a_j,b_j\}$"=полулинейной 
по~$j$-му аргументу}, $j=0,1,\ldots,n$. 
Заметим, что если в тождестве (\ref{eq:semilinear}) любые два из множеств $\{a_j,b_j\}$ заменить на их дополнения в $\Sigma$, 
то тождество останется верным. 
Таким образом, по каждому аргументу полулинейная квазигруппа является $\{0,1\}$, $\{0,2\}$ или $\{0,3\}$"=полулинейной.
%[??? может упростить запись: $ab$"=полулинейная? или нет...]
Если $f$ является $\{a,b\}$"=полулинейной по каждому своему аргументу, 
то будем называть ее просто \textit{$\{a,b\}$"=полулинейной.} 

Квазигруппа $f$ \textit{линейна}, если 
по~каждому своему аргументу она $\{a,b\}$"=полулинейна для любых $a,b\in\Sigma$. 

Каждая $2$-квазигруппа изотопна одной из двух квазигрупп $\oplus$, $+_4$ с таблицами значений 
\begin{equation} \label{eq:Z22,Z4}
\begin{array}{c|cccc}
    \oplus & 0 & 2 & 1 & 3\\
    \hline
    0 & 0 & 2 & 1 & 3\\
    2 & 2 & 0 & 3 & 1\\
    1 & 1 & 3 & 0 & 2\\
    3 & 3 & 1 & 2 & 0
\end{array}, \qquad
\begin{array}{c|cccc}
    +_4 & 0 & 2 & 1 & 3\\
    \hline
    0 & 0 & 2 & 1 & 3\\
    2 & 2 & 0 & 3 & 1\\
    1 & 1 & 3 & 2 & 0\\
    3 & 3 & 1 & 0 & 2
\end{array}.
\end{equation}
Квазигруппы $(\Sigma,\oplus)$ и $(\Sigma,+_4)$ суть группы $\mathbb{Z}_2\times\mathbb{Z}_2$ и $\mathbb{Z}_4$, соответственно (напомним, группа это ассоциативная $2$"=квазигруппа с нейтральным элементом). 
\begin{note} В таблицах значений (\ref{eq:Z22,Z4}) элементы $0$, $2$, $1$, $3$ упорядочены не лексикографически, а по возрастанию порядка в группе $\mathbb{Z}_4$. 
Таким образом легче наблюдать полулинейную структуру данной группы. 
В дальнейшем, аналогично, таблицы значений $\{0,2\}$"=полулинейных $n$"=квазигрупп удобно мыслить $n$"=мерными массивами $4\times\ldots\times 4$, 
которые разбиты на $n$"=мерныe подмассивы $2\times\ldots\times 2$, заполненные двумя значениями $0$, $2$ или $1$, $3$.
\end{note}
Итерированную $n-1$ раз квазигруппу $\oplus$ обозначим через $l_n$, то есть
$$ l_n(x_1,\ldots,x_n) = x_1\oplus\ldots\oplus x_n. $$

\begin{lemma}[\cite{RU:PotKro:asymp.ru}]\label{l:linear}
(i) Все линейные $n$"=квазигруппы изотопны $l_n$. 
(ii) Если $n$"=квазигруппа является одновременно $\{0,1\}$- и $\{0,2\}$"=полулинейной по некоторому своему аргументу, то она линейна.
\end{lemma}

\begin{lemma}\label{l:isemilinear}
Если $n$"=квазигруппа $f$ является $\{a,b\}$"=полулинейной по $i$-му аргументу для некоторого 
$i\in\{0,...,n\}$ и~$\theta=(\theta_0,\ldots,\theta_n)$~--- изотопия, то $n$"=квазигруппа $\theta(f)$ является 
$\{\theta_i^{-1}(a),\theta_i^{-1}(b)\}$"=полулинейной по $i$-му аргументу.
\end{lemma}

Квазигруппа $f$ арности $n$ называется \textit{разделимой}, если для некоторого целого числа 
$m,\,2\leqslant m<n$, и перестановки $\sigma\in{S_n}$ справедливо представление $f$ в виде бесповторной
композиции 
\begin{equation}
%\label{eq:composition}
f(x_1,\ldots,x_n)= h(g(x_{\sigma(1)},\ldots,x_{\sigma(m)}),x_{\sigma(m+1)},\ldots,x_{\sigma(n)})
\end{equation}
(бесповторность означает, что каждая переменная встречается в правой части только один раз).
Без ограничения общности, можем полагать, что квазигруппа $g$ неразделима.

В~\cite{RU:KroPot:4} получено описание квазигрупп порядка $4$ в терминах полулинейности и разделимости.

\begin{theorem}
\label{th:top}
Произвольная $n$-арная квазигруппа порядка $4$ является разделимой или полулинейной. 
\end{theorem}

%=================================
%=================================
%=================================
\section{Представление квазигрупп}\label{s:repres}

\hspace*{\parindent} Согласно теореме~\ref{th:top}, квазигруппа, не являющаяся полулинейной, 
может быть представлена в виде бесповторной композиции двух или более полулинейных квазигрупп (в числе
которых могут оказаться одинаковые квазигруппы, равно как и линейные). Запись квазигруппы $f$ 
в виде бесповторной композиции квазигрупп арности больше $1$ будем называть \textit{разложением} $f$. Отметим, что квазигруппа 
может иметь несколько разложений. В простейшем случае, квазигруппа представляет свое 
собственное тривиальное разложение.
% Оговорим, что все рассматриваемые ниже квазигруппы имеют арность $2$ и более.

В дальнейшем будем использовать графическое представление разложения квазигруппы в виде помеченного дерева.
Внутренние вершины этого дерева (степень которых не меньше $3$) будем называть \textit{узлами} и обозначать 
символами %входящих в разложение квазигрупп 
$u,v,w$, с индексами или без них, 
а листья (вершины степени $1$)~--- символами переменных $x_1,x_2,\ldots,y,z$.
%, ассоциированных с сооответствующими аргументами квазигруппы. 
Ребро, инцидентное листу дерева, назовем \textit{листовым}. Остальные ребра будем называть 
\textit{внутренними}. 

Сначала, определим рекурсивно \textit{корневое дерево} $T(S)$ разложения $S$
(понятие корневого дерева разложения вводится как вспомогательный термин 
для определения дерева разложения и после этого использоваться не будет).
\par
1) Переменной $x_i$ поставим в соответствие дерево, состоящее из одной вершины степени $0$, 
являющейся корнем и помеченной самой переменной $x_i$. 
% Квазигруппе $f(x_1,x_2)$ соответствует дерево, состоящее из корня $u$ с меткой $f$ и двух листьев $x_1$ и $x_2$, смежных с корнем.
\par
2) Пусть разложение $S$ имеет вид $S=h(S_1,\ldots,S_n)$. Если разложениям и/или переменным $S_1,\ldots,S_n$ 
соответствуют корневые деревья $T_1,\ldots,T_n$, то дерево $T(S)$ строим следующим образом. 
Корнем дерева определим новую вершину $u$ и припишем ей метку $h$. Последовательно соединим 
вершину $u$ с корнями деревьев $T_1,\ldots,T_n$. При этом будем считать, что корень дерева 
$T_j$, $j\in\{1,...,n\}$, является $j$-м соседом вершины $u$. С другой стороны, вершину $u$ будем считать 
нулевым соседом корня дерева $T_j$.

\textit{Деревом разложения} (не корневым) будем называть дерево, полученное присоединением 
листа $x_0$ в качестве нулевого соседа к корню дерева $T(S)$. Дерево без корня для разложения $S$ 
обозначим через $T_0(S)$. Лист $x_0$ соответствует нулевому аргументу, то есть значению квазигруппы, 
которую представляет разложение $S$.

Дерево $T_0(S)$ разложения $S$ квазигруппы $f$ можно также мыслить как дерево разложения для кода $M(f)$. 
При этом важно понимать, что только нумерация листьев и соседей каждой внутренней вершины определяет, какие аргументы являются независимыми
для квазигруппы $f$ или каждого элемента разложения. Сменив эту нумерацию, мы можем получить дерево разложения для обращения $f$ по любому из аргументов.
А именно, чтобы получить разложение для $f^{<i>}$, достаточно для каждой внутренней вершины пути из $x_0$ в $x_i$ поменять между собой номера двух соседей и
заменить метку вершины на соответствующее обращение, а также поменять между собой метки листьев $x_i$ и $x_0$. 
Порядок, определенный на множестве 
соседей каждого узла, однозначно задает порядок аргументов входящих в разложение квазигрупп и квазигруппы, представленной разложением. 
Отметим, что поскольку группы автотопий квазигруппы и ее обращений изоморфны, с точки зрения вопросов, рассматриваемых в настоящей работе, 
нам не всегда необходимо помнить, какой именно из аргументов является нулевым, и во многих рассуждениях мы не будем его отдельно выделять.
%Лист $x_0$ позволяет восстановить корень дерева разложения, а 

Для разложения и соответствующего ему дерева определим операцию \textit{стягивания.} 
Предположим, разложение $S$ содержит фрагмент
\begin{equation}
\label{eq:fragment}
	f_1(S_1,\ldots,S_{i-1},f_2(S_i,\ldots,S_{i+n_2-1}),S_{i+n_2},\ldots,S_{n_1+n_2-1}),
\end{equation}
где $S_1,\ldots,S_{n_1+n_2-1}$~--- некоторые разложения, а $n_1$"=квазигруппа $f_1$ 
и $n_2$"=квазигруппа $f_2$ 
%на множестве $\Sigma^{n_1+n_2-1}$, 
удовлетворяют тождеству
\begin{equation}
\label{eq:quasifragment}
 g(x_1,\ldots,x_{n_1+n_2-1})\equiv f_1(x_1,\ldots,x_{i-1},f_2(x_i,\ldots,x_{i+n_2-1}),x_{i+n_2},\ldots,x_{n_1+n_2-1})
\end{equation}
для некоторой $(n_1+n_2-1)$"=квазигруппы $g$.
%и попарно различных переменных $x_1,\dots,x_{n_1+n_2-1}$. 
В этом случае 
результатом стягивания квазигрупп $f_1$ и $f_2$ в разложении $S$ определим разложение $\widetilde{S}$, 
полученное из $S$ заменой фрагмента~\eqref{eq:fragment} на $g(S_1,\ldots,S_n)$. 
Подчеркнем, что речь идет о конкретных вхождениях квазигрупп $f_1$ и $f_2$ в разложение $S$, 
в котором фрагмент вида~\eqref{eq:fragment} может встречаться неоднократно.

Соответственно стягиваются смежные вершины $u$ и $v$ в дереве разложения $T_0(S)$, 
помеченные метками $f_1$ и $f_2$. Указанная пара вершин заменяется на новую вершину~$w$ 
с меткой $g$, соседями которой становятся все соседи удаленных вершин. Соседям 
вершины $w$ последовательно присваиваются номера $0,\ldots,n_1+n_2-1$ в следующем порядке. 
Сначала нумеруются соседи вершины $u$ с номерами $0,\ldots,i-1$ (с сохранением порядка), 
затем~--- соседи вершины $v$ с номерами $1,\ldots,n_2$, после чего~--- оставшиеся соседи
вершины $u$ с номерами $i+1,\ldots,n_1$. В результате получается дерево разложения~$T_0(\widetilde{S})$. Тривиально, имеем:

\begin{lemma}\label{l:styag}
Стягивание не изменяет квазигруппу, которую представляет разложение.
\end{lemma}

Разложение (и его дерево) назовем \textit{полулинейным,} если все входящие в него квазигруппы 
полулинейны.

Рассмотрим в дереве разложения два соседних узла $u,v$ с метками $f_1,f_2$. Допустим при этом, что 
$u$ является нулевым соседом $v$, а $v$ является $i$-м соседом $u$. Узлы $u$ и $v$ назовем 
\emph{согласованными}, если для некоторых $a,b\in\Sigma$ квазигруппа $f_1$ является 
$\{a,b\}$"=полулинейной по $i$-му аргументу, а $f_2$~--- $\{a,b\}$"=полулинейной по нулевому.

\begin{lemma}\label{l:soglas}
При стягивании двух согласованных узлов в полулинейном дереве оно остается полулинейным.
\end{lemma}

\begin{proof} Пусть квазигруппы $f_1$ и $f_2$ арности $n_1$ и $n_2$ соответствуют согласованным 
узлам в дереве разложения, причем для некоторой квазигруппы $g$ арности $n_1+n_2-1$ на множестве 
$\Sigma^{n_1+n_2-1}$ справедливо~(\ref{eq:quasifragment}).
Для доказательства леммы достаточно убедиться, что квазигруппа $g$ полулинейна. Это нетрудно сделать, 
используя указанное тождество и определение полулинейной квазигруппы.
\end{proof}

Полулинейное разложение (и его дерево) назовем \emph{правильным}, если в дереве этого разложения 
отсутствуют пары согласованных узлов.

\begin{lemma}\label{l:prav}
Произвольная квазигруппа порядка $4$ имеет правильное разложение.
\end{lemma}

\begin{proof}
По теореме~\ref{th:top} у любой $n$-арной квазигруппы порядка $4$ имеется полулинейное разложение. 
Поскольку в дереве такого разложения не более $n-1$ узлов, последовательно стягивая пары согласованных, 
не более чем за $n-2$ шагов получим правильное разложение.
\end{proof}

\begin{note}
Правильное разложение в общем случае не единственно и зависит от порядка стягивания. 
Простейшим примером являются два способа стянуть разложение $f(g(h(x_1,x_2),x_3),x_4)$, в котором $f$ и $g$ суть $\{0,1\}$"=полулинейные квазигруппы, 
а $g$  и $h$ суть $\{0,2\}$"=полулинейные квазигруппы, при этом $f$ и $h$ не являются линейными, в отличие от $g$.
В правильном разложении 
нелинейной квазигруппы отсутствуют линейные квазигруппы, поскольку узел, помеченный линейной квазигруппой, согласован с каждым из соседних узлов.
\end{note}

Пусть $S$~--- некоторое разложение и $T_0(S)$~--- дерево этого разложения с множеством ребер $E$. 
\emph{Изотопией разложения} назовем набор перестановок $\theta=(\theta_e)_{e\in{E}}$ множества~$\Sigma$, 
действующий на разложение $S$ следующим образом. Если узел в дереве~$T_0(S)$ имеет метку~$f_i$ 
и ребро~$e_j$, $j=0,1,\ldots,n_i$,~--- $j$-е ребро, инцидентное этому узлу, то $f_i$ заменяется 
на метку $f'_i$, где $f'_i$ --- квазигруппа, определяемая тождеством
\begin{equation}\label{eq:isot_i}
f'_i(x_1,\ldots,x_{n_i}) = \theta_{e_0}^{-1}f_i(\theta_{e_1} x_1,\ldots,\theta_{e_{n_i}} x_{n_i}).
\end{equation}
%Здесь и далее точка $\cdot$ обозначает операцию произведения функций, при этом, например, $p\cdot q(t)=q(p(t))$. 
В результате получим дерево некоторого разложения $\theta(S)$. Очевидно следующее утверждение.

\begin{lemma}\label{l:isotop}
Изотопные разложения представляют изотопные квазигруппы. Точнее, если $\theta$~--- изотопия 
разложений квазигрупп $f$ и $f'$, то
$$x_0=f'(x_1,\ldots,x_{n}) = \theta_{e_0}^{-1}f(\theta_{e_1}(x_1),\ldots,\theta_{e_{n}}(x_{n})),$$
где $e_j$~--- ребро, инцидентное листу $x_j$, $j=0,1,\ldots,n$.
\end{lemma}

\emph{Автотопией} разложения $S$ назовем изотопию $\theta$ такую, что $\theta(S)=S$. \emph{Носителем}
автотопии назовем множество ребер, которым сопоставлена нетождественная перестановка.
Правильное разложение (и его дерево) назовем \emph{приведенным}, если каждая входящая в него квазигруппа 
$\{0,1\}$- или $\{0,2\}$"=полулинейна.

\begin{lemma}\label{l:dobr}
Для произвольной квазигруппы порядка $4$ найдется изотопная ей квазигруппа с приведенным разложением.
\end{lemma}

\begin{proof} Рассмотрим $n$"=квазигруппу $f$ и построим изотопную ей квазигруппу, обладающую приведенным 
разложением. Начнем с правильного разложения $S$ квазигруппы $f$, которое существует по лемме~\ref{l:prav}. 
Множество вершин дерева разложения $T_0(S)$, как двудольного графа, разобьем на две доли; вершины одной доли назовем четными, 
другой~--- нечетными.

Подберем такую изотопию $\theta$, что нечетные узлы разложения $\theta(S)$ будут $\{0,1\}$"=полулинейными, 
а четные~--- $\{0,2\}$"=полулинейными. Для этого определим перестановку~$\theta_e$ для каждого ребра $e$ 
в дереве $T_0(S)$.

Рассмотрим два случая. Пусть сначала $e$ соединяет два узла~--- нечетный с меткой~$g$ и четный 
с меткой $h$. Допустим, что $g$ является $i$-м соседом $h$, который в~свою очередь является нулевым соседом $g$. 
Заметим, что если квазигруппа $g$ является $\{0,a\}$"=полулинейной по нулевому аргументу, 
а квазигруппа $h$ является $\{0,b\}$"=полулинейной по $i$-му, то $a\ne{b}$, поскольку разложение 
$S$ правильное. В этом случае положим $\theta_e(0)=0$, $\theta_e(1)=a$, $\theta_e(2)=b$, 
$\theta_e(3)\in\{1,2,3\}\setminus\{a,b\}$.

Теперь обратимся ко второму случаю, когда одна из вершин ребра $e$ является узлом и отмечена 
квазигруппой $g$, а вторая является $i$-м соседом первой и представляет собой переменную $x$. 
Допустим, $g$ является $\{0,a\}$"=полулинейной по $i$-му аргументу. Тогда положим $\theta_e=(1a)$, 
если рассматриваемый узел входит в нечетную долю дерева~$T_0(S)$, и $\theta_e=(2a)$, если в четную.

Рассмотрим действие построенной изотопии на разложение $S$. Узел дерева разложения $T_0(S)$, 
помеченный квазигруппой $f_i$, в дереве $T_0(\theta(S))$ будет помечен квазигруппой $f'_i$ из 
(\ref{eq:isot_i}). Предположим, что $f_i$ является $\{0,a\}$"=полулинейной по $j$-му аргументу. Тогда 
по лемме~\ref{l:isemilinear} квазигруппа $f'_i$ будет $\{0,1\}$"=полулинейной по $j$-му аргументу, 
если $f_i$~--- нечетный узел, и $\{0,2\}$"=полулинейной по $j$-му аргументу, если $f_i$~--- четный узел.

Таким образом, разложение $\theta(S)$ является приведенным по определению. По лемме~\ref{l:isotop} квазигруппа, 
которую представляет это разложение, изотопна исходной квазигруппе $f$.
\end{proof}

%\begin{lemma}\label{l:odnot}
%Для каждой нелинейной $n$-арной квазигруппы $f$ найдется другая $n$-арная квазигруппа $f'$ такая, что
%$\mathrm{Atp}(f') \le \mathrm{Atp}(f)$ и $f'$ имеет доброе разложение, состоящее из одноточечных [!!!!!термин] квазигрупп (далее будем называть такое разложение одноточечным).
%\end{lemma}

%=================================
%=================================
%=================================
\section{Автотопии квазигрупп}\label{s:au}

\hspace*{\parindent} Пусть $\pi=(\pi_0,\ldots,\pi_m)$ и $\tau=(\tau_0,\ldots,\tau_{n-m+1})$~--- 
изотопии.
%множеств $\Sigma^{m+1}$ и $\Sigma^{n-m+2}$ соответственно. 
Если $\pi_0=\tau_1$, то положим 
$$\pi\dot{\otimes}\tau=(\tau_0,\pi_1,\ldots,\pi_m,\tau_2,\ldots,\tau_{n-m+1}).$$

Рассмотрим $n$-арную квазигруппу $f$, образованную композицией $m$"=квазигруппы~$g$ и $(n-m+1)$"=квазигруппы $h$:
$$f(x_1,\ldots,x_n)= h(g(x_1,\ldots,x_m),x_{m+1},\ldots,x_n).$$
На группах автотопий квазигрупп $g$ и $h$ операцию $\dot{\otimes}$ 
определим следующим образом:
$$
	\mathrm{Atp}(g)\dot{\otimes}\mathrm{Atp}(h)=
	\{\pi\dot{\otimes}\tau\mid \pi=(\pi_0,\ldots,\pi_m)\in\mathrm{Atp}(g),\tau=(\tau_0,\ldots,\tau_{n-m+1})\in\mathrm{Atp}(h), \pi_0=\tau_1\}.
$$
Хотя мы уже ограничили себя рассмотрением квазигрупп только порядка $4$, следующая лемма верна для произвольного порядка.

\begin{lemma}
\label{lem:composition}
Если $f$ --- $n$"=квазигруппа, представимая в виде композиции $f(x_1,\ldots,x_n)= h(g(x_1,\ldots,x_m),x_{m+1},\ldots,x_n)$, то 
$$\mathrm{Atp}(f)=\mathrm{Atp}(g)\dot{\otimes}\mathrm{Atp}(h).$$
\end{lemma}

\begin{proof}
Очевидно, что $\mathrm{Atp}(g)\dot{\otimes}\mathrm{Atp}(h)\leqslant\mathrm{Atp}(f)$. Чтобы доказать обратное, 
рассмотрим автотопию $\theta=(\tau_0,\pi_1,\ldots,\pi_m,\tau_2,\ldots,\tau_{n-m+1})\in\mathrm{Atp}(f)$. Покажем, 
что существует перестановка $\pi_0=\tau_1\in{S_4}$ такая, что 
$\pi=(\pi_0,\ldots,\pi_m)\in\mathrm{Atp}(g)$, 
$\tau=(\tau_0,\ldots,\tau_{n-m+1})\in\mathrm{Atp}(h)$ и 
$\theta=\pi\dot{\otimes}\tau$.

Заметим, что если указанная перестановка $\pi_0$ существует, то она однозначно определяется перестановками $\pi_1$, \ldots, $\pi_m$, поскольку 
для каждой вершины в $\Sigma^n$ квазигруппа $g$ принимает одно значение. Кроме того, если положить 
$\pi_0=\tau_1$, то $\pi\in\mathrm{Atp}(g)$ тогда и только тогда, когда 
$\tau\in\mathrm{Atp}(h)$. Действительно, отношение $\pi\in\mathrm{Atp}(g)$ по определению 
означает, что равенства
\begin{multline*}
	x_0=h(g(x_1,\ldots,x_m),x_{m+1},\ldots,x_n) \text{ и }\\
	x_0=h(\pi_0 g(\pi_1^{-1}x_1,\ldots,\pi_m^{-1}x_m),x_{m+1},\ldots,x_n)
\end{multline*}
выполняются одновременно. Отсюда применением автотопии $\theta\in\mathrm{Atp}(f)$ для любого набора 
$(x_1,\ldots,x_n)\in{\Sigma^n}$ получаем
\begin{multline*}
	x_0=h(g(x_1,\ldots,x_m),x_{m+1},\ldots,x_n)=\\
	=\tau_0^{-1}h(\pi_0g(x_1,\ldots,x_m),\tau_2 x_{m+1},\ldots,\tau_{n-m+1}x_n).
\end{multline*}
Последнее равенство эквивалентно тому, что для любого $(t,x_{m+1},\ldots,x_n)\in{\Sigma^{n-m+1}}$ выполнено
$$
	h(t,x_{m+1},\ldots,x_n)=\tau_0^{-1}h(\pi_0 t,\tau_2 x_{m+1},\ldots,\tau_{n-m+1}x_n),
$$
то есть при $\tau_1=\pi_0$ имеем $\tau\in\mathrm{Atp}(h)$.

Таким образом, остается показать, что существует перестановка $\pi_0\in{S_4}$, для которой 
$\pi\in\mathrm{Atp}(g)$. Ввиду условия $\theta\in\mathrm{Atp}(f)$, можем записать в частности, 
что для любого ${(x_1,\ldots,x_m)}\in{\Sigma^m}$ справедливы равенства
\begin{equation}
\label{eq:zeroes}
	x_0=h(g({x}),0,\ldots,0)=\tau_0^{-1}h(g(\pi_1 x_1,\ldots,\pi_m x_m),\tau_2(0),\ldots,\tau_{n-m+1}(0)).
\end{equation}

Очевидно, что $1$"=квазигруппы 
$$q_1(s)=h(s,0,\ldots,0),\quad q_2(t)=\tau_0^{-1}h(t,\tau_2(0),\ldots,\tau_{n-m+1}(0))$$ 
суть перестановки из $S_4$. Таким образом, (\ref{eq:zeroes}) можно переписать в виде:
$$g({x})=q_1^{-1} ( q_{2}( g  (\pi_1x_1,\ldots,\pi_mx_m))) .$$
Определив $\pi_0(\cdot)=q_2^{-1}(q_1(\cdot)) $, имеем $(\pi_0,\ldots,\pi_m))\in\mathrm{Atp}(g)$.
\end{proof}

Из леммы~\ref{lem:composition} и результатов предыдущего раздела следует, что при изучении группы 
автотопий квазигруппы можно считать, что она представима в виде бесповторной композиции квазигрупп, каждая из которых
$\{0,a\}$"=полулинейна для некоторого $a\in\Sigma$, но не линейна.

В оставшейся части данного раздела мы докажем три леммы о минимальной группе автотопий полулинейной квазигруппы.
При описании автотопий удобно использовать следующие обозначения.

Для нелинейной $\{0,a\}$"=полулинейной квазигруппы, $a\in\Sigma\setminus\{0\}$, (и~соответствующего ей узла)
перестановку $(0a)(bc)$, где $\{b,c\}=\Sigma\setminus\{0,a\}$ назовем \emph{родной инволюцией}, 
а $(0b)(ca)$ и $(0c)(ab)$ -- \textit{неродными инволюциями}.

Каждую из транспозиций $(0a)$ и $(bc)$, образующих родную инволюцию $(0a)(bc)$, назовем 
\emph{родной транспозицией} полулинейной квазигруппы (узла).

Две циклические перестановки $(0bac)$ и $(0cab)$, квадрат которых равен родной инволюции $(0a)(bc)$, 
назовем \emph{родными циклами} полулинейной квазигруппы (узла).

\begin{lemma}\label{l:2rodn}
Следующие изотопии принадлежат группе автотопий 
$\{0,a\}$"=полулинейной $n$-арной квазигруппы $f$, 
где $a\in\Sigma\setminus\{0\}$.
\par
{\rm (i)}\;Изотопия, в произвольном порядке составленная из двух родных инволюций и~$n-1$ тождественных 
перестановок.\par
{\rm (ii)}\;Изотопия из $n+1$ родных транспозиций, отличных от $(0a)$, если $f(\{0,a\}^n)=\{0,a\}$,
или изотопия из $n+1$ родных транспозиций, ровно одна из которых совпадает с $(0a)$, если $f(\{0,a\}^n)=\Sigma\backslash\{0,a\}$.
\end{lemma}
\begin{proof} Без потери общности положим $a=1$.
Тождество (\ref{eq:semilinear}) выполнено для любых $\{a_i,b_i\}$ из $\{\{0,1\},\{2,3\}\}$, $i=1,\ldots,n$
(пара $\{a_0,b_0\}$ определяется однозначно по всем остальным и также совпадает с $\{0,1\}$ или $\{2,3\}$).

(i) Применение родной инволюции $(01)(23)$ в одном из аргументов изменяет значения квазигруппы во всех точках,
но при этом оставляет множества $\{a_1,b_1\}\times\ldots\times\{a_n,b_n\}$ с указанными выше ограничениями на месте.
Из  (\ref{eq:semilinear}) следует, что значения квазигруппы  меняются также в соответствии с родной инволюцией.
При применении родной инволюции в некотором другом аргументе мы опять получим исходную квазигруппу.

(ii) Пусть $f(\{0,1\}^n)=\{0,1\}$. Рассмотрим произвольный набор значений аргументов $(x_1,\ldots,x_n)$ и значение квазигруппы $x_0$ на этом наборе.
Среди $x_0$, $x_1$, \ldots, $x_n$ четное число значений из $\{2,3\}$. Таким образом, применяя последовательно транспозицию $(23)$ к каждому из аргументов,
мы изменим значение квазигруппы четное число раз, при этом изменения не выводят значение за рамки пары $\{0,1\}$ или пары $\{2,3\}$.
В итоге мы получаем, что после применения всех транспозиций значение квазигруппы не изменилось. Случай $f(\{0,1\}^n)=\{2,3\}$ рассматривается аналогично.
\end{proof}

\begin{lemma}\label{l:aut3}
Пусть $\{a,b\}$-полулинейная бинарная квазигруппа $q$ порядка $4$ не является линейной и имеет родную  
инволюцию $\xi$. Группа автотопий квазигруппы~$q$ исчерпывается следующими преобразованиями.\par
{\rm (i)}\;Автотопии $(\mathrm{Id},\xi,\xi),$ $(\xi,\mathrm{Id},\xi),$ $(\xi,\xi,\mathrm{Id})$ и~тождественная автотопия.\par
{\rm (ii)}\;Автотопии $(\tau_0,\tau_1,\tau_1),$ $(\tau_1,\tau_0,\tau_1),$ $(\tau_1,\tau_1,\tau_0),$ $(\tau_0,\tau_0,\tau_0),$ 
где $\tau_0$,  $\tau_1$~--- две различные родные транспозиции квазигруппы $q$, где выбор $\tau_0$ однозначен для данной квазигруппы. \par
{\rm (iii)}\;Автотопии $(\xi',\varphi_1,\varphi_2),$ $(\varphi_1,\xi'',\varphi_2),$ $(\varphi_1,\varphi_2,\xi'''),$ 
где $\varphi_1,\varphi_2$~--- произвольная пара родных циклов, для которой единственным образом определяются 
перестановки $\xi',\xi'',\xi'''\in\{\mathrm{Id},\xi\}.$\par
{\rm (iv)}\;Автотопии $(\tau',\psi_1,\psi_2),$ $(\psi_1,\tau'',\psi_2),$ $(\psi_1,\psi_2,\tau'''),$ 
где $\psi_1,\psi_2$~--- произвольная пара неродных инволюций, для которой единственным образом определяются 
подходящие родные транспозиции $\tau',\tau'',\tau'''.$
\end{lemma}
\begin{proof}
 То, что каждая из представленных изотопий является автотопией, проверяется непосредственно.
 Для этого рассмотрим  $\{0,2\}$-полулинейную бинарную квазигруппу $+_4$ (см. пример~\ref{ex:1} ниже), поскольку все квазигруппы, удовлетворяющие условию леммы, изотопны ей. Легко видеть, что множество представленных автотопий замкнуто, то есть образует группу.
 
 Полноту  проверим численно. Имеется по $4$ автотопии типа (i) и (ii), по $12$ автотопий типа (iii) и (iv),
 то есть всего $32$ автотопии. Оценим число автотопий с другой стороны.
 Из нелинейности  квазигруппы $+_4$ и леммы~\ref{l:isemilinear} следует, что для произвольной автотопии $(\psi_0,\psi_1,\psi_2)$ 
 каждая из перестановок отображает $\{0,2\}$ в $\{0,2\}$ или $\{1,3\}$.
 Существует по $8$ способов выбрать $\psi_1$ и $\psi_2$, удовлетворяющие указанному условию;
 по определению квазигруппы $\psi_0$ задается однозначно.
 Более того, легко проверить, что не существует автотопии с $\psi_1=\mathrm{Id}$ и $\psi_2=(01)$.
 Таким образом, группа автотопий имеет порядок меньше $64$, и следовательно,  совпадает с группой, состоящей из автотопий (i)--(iv). 
\end{proof}

\begin{example}\label{ex:1}
 Рассмотрим бинарную квазигруппу $+_4$, определенную в (\ref{eq:Z22,Z4}).
Перестановка $(02)(13)$ является родной инволюцией квазигруппы $q$, перестановки $(02)$ и $(13)$~--- 
родные транспозиции для $q$, а перестановки $(0123)$ и~$(0321)$~--- родные циклы. Группа автотопий 
квазигруппы $q$ порождается следующим множеством автотопий (строго говоря, избыточным):\par
(i) $((02)(13),(02)(13),\mathrm{Id})$, $((02)(13),\mathrm{Id},(02)(13))$, $(\mathrm{Id},(02)(13),(02)(13))$;\par
(ii) $((13),(13),(13))$;\par
(iii) $(\mathrm{Id},(0123),(0321))$;\par
(iv) $((13),(03)(12),(01)(23))$, $((02),(01)(23),(01)(23))$, $((02),(03)(12),(03)(12))$.
\end{example}

Таким образом, мы знаем группу автотопий единственной, с точностью до изотопии, нелинейной бинарной квазигруппы.
Кроме этого, нам понадобятся примеры полулинейных $3$- и $4$-арных квазигрупп с минимальной группой автотопий.
Определим  $n$-арную квазигруппу $l_n^\bullet$ тождеством
$$l_n^\bullet(x)= \begin{cases}
                    l_n(x)\oplus 2,& \mbox{если $x\in\{0,2\}^n$,}  \cr l_n(x), & \mbox{если $x\not\in\{0,2\}^n$.} 
                       \end{cases}
$$

\begin{lemma}\label{l:aut4}
 При $n\ge 3$ группа автотопий $n$-квазигруппы $l_n^\bullet$ порождается автотопиями из леммы~\ref{l:2rodn} и имеет порядок $2^{n+1}$.
\end{lemma}
Перед тем как привести доказательство, заметим, что в дальнейшем нам понадобятся только случаи $n=3$ и $n=4$, 
для которых можно провести доказательство прямой проверкой, ручной или машинной.
\begin{proof}
 Очевидно, что автотопии из леммы~\ref{l:2rodn} имеют порядок $2$, коммутируют и линейно независимы, откуда следует порядок порождаемой ими группы.
 
 Код $M(l_n)$ квазигруппы $l_n$ является $2n$-мерным аффинным подпространством векторного пространства над полем GF$(2)$ из двух элементов со сложением $\oplus$ 
 и тривиальным умножением на нулевой и единичный элементы поля.
 Код $M(l_n^\bullet)$ квазигруппы $l_n^\bullet$ отличается 
 от аффинного подпространства $M(l_n)$ в $2^n$ вершинах множества $B_n$, 
 где $B_n=M(l_n^\bullet)\backslash M(l_n) = \{(l_n^\bullet (x),x)\mid x\in\{0,2\}^n\}$.
 Более того, $M(l_n)$ --- единственное ближайшее (в указанном выше смысле) к $M(l_n^\bullet)$ аффинное подпространство, 
 поскольку любое другое аффинное подпространство той же размерности отличается от $M(l_n)$ не менее чем в $2^{2n-1}\ge 4\cdot 2^n$ вершинах.
 Под действием автотопии квазигруппы $l_n^\bullet$ код  $M(l_n^\bullet)$ переходит сам в себя (по определению), а $M(l_n)$ --- в аффинное подпространство
 (действительно, легко убедиться, что любая перестановка множества $\Sigma$ является аффинным преобразованием над GF$(2)$), 
 которое также является ближайшим к $M(l_n^\bullet)$. 
 Таким образом, автотопия квазигруппы $l_n^\bullet$ является автотопией линейной квазигруппы $l_n$. 
 Кроме того, отсюда следует, что под действием такой автотопии множество $B_n$ (разность между кодами этих квазигрупп) переходит само в себя.
 В частности, каждая перестановка этой автотопии стабилизирует $\{0,2\}$, то есть является одной из $\mathrm{Id}$, $(02)$, $(13)$, $(02)(23)$.
 Как следует из описания группы автотопий $l_n$ в разделе~\ref{s:up}, все такие автотопии являются комбинацией автотопий из леммы~\ref{l:2rodn}.
\end{proof}

%=================================
%=================================
%=================================
\section{Нижняя оценка}\label{s:lb}

В этом разделе рассмотрим произвольную квазигруппу порядка $4$ и докажем достижимую нижнюю оценку 
на порядок ее группы автотопий. В частности, полулинейные квазигруппы имеют достаточно большие группы автотопий. Для разделимой квазигруппы будет показано, что узлы дерева разложения можно сгруппировать по группам, которые мы называем гроздьями (каждая гроздь состоит из узлов одной четности, то есть в ней нет соседних узлов); каждой грозди соответствует подгруппа группы автотопий квазигруппы, причем подгруппы, соответствующие разным гроздьям, независимы.

Введем дополнительные обозначения и определения, 
связанные с представлением квазигрупп в виде дерева разложения. Пусть $f$~--- $n$-арная квазигруппа 
порядка $4$, имеющая приведенное разложение $S$ с деревом разложения $T=T_0(S)$.\par
$\bullet$\;Положим $N=n+1$~--- число листьев в дереве $T$, а через $V$ 
обозначим число узлов в нем.

$\bullet$\;\emph{Лысый} узел~--- внутренняя вершина дерева $T$, среди соседей которой отсутствуют листья. 
Число лысых узлов в $T$ обозначим через $E$.

$\bullet$\;\emph{Мост}~--- вершина степени $3$, смежная ровно с одним листом. Лист, смежный с~мостом, 
назовем \emph{мостовым}. Число мостов в $T$ обозначим через $B$.

$\bullet$\;\emph{Вилка}~--- вершина степени $3$, смежная ровно с двумя листьями. 
Число вилок в~дереве $T$ обозначим через $F$.

$\bullet$\;$G(T)$ --- граф, в котором вершинами являются узлы дерева $T$ и две вершины смежны тогда и 
только тогда, когда они смежны с одним и тем же мостом в дереве~$T$. Легко видеть, что $G(T)$~--- лес.

$\bullet$\;\emph{Гроздь}~--- компонента связности графа $G(T)$. \emph{Листьями грозди} назовем листья дерева $T$, 
смежные с узлами этой грозди. Гроздь назовем \emph{лысой}, если у нее нет листьев, то есть все ее узлы лысые. 
Число гроздей обозначим через $\mathit\Gamma$, а число лысых гроздей~--- через $L$.

Заметим, что мосты, порождающие ребра некоторой грозди, сами не являются её узлами, 
но распределены по другим гроздьям. Более того, все узлы грозди принадлежат одной доле 
двудольного графа $T$, а мосты, соответствующие рёбрам этой грозди,~--- другой доле.

Для примера рассмотрим дерево разложения, изображенное на рис.~\ref{fig1}. В нем имеются лысый узел $\iota$, 
$5$ мостов $\gamma,\delta,\zeta,\eta,\theta$ и одна вилка $\beta$. Узлы группируются по~гроздьям следующим образом: 
$\{\alpha,\beta,\varepsilon,\eta,\iota\}$, $\{\gamma\}$, $\{\delta\}$, $\{\zeta,\theta\}$, $\{\kappa\}$, $\{\lambda\}$, $\{\mu\}$.
 \tikzset{x=0.6cm,y=0.6cm,
	nodd/.style={rectangle,fill=white, draw=black, minimum size=0.5cm, rounded corners=4pt},
	leaf/.style={circle,fill=white, draw=black},
}
\begin{figure}[h]
$$
\begin{tikzpicture}
\draw (0,0) node[nodd,fill=red!50]{$\varepsilon$} -- 
++(-15:2) node[nodd,fill=yellow!60]{$\zeta$} -- 
 +(-90:1) node[leaf]{} +(0,0) --
++(15:2) node[nodd,fill=red!50]{$\eta$} --
 +(90:1) node[leaf]{} +(0,0) --
++(-15:2) node[nodd,fill=yellow!60]{$\theta$} --
 +(-90:1) node[leaf]{} +(0,0) --
++(15:2) node[nodd,fill=red!50]{$\iota$} --
 %+() +(0,0) --
++(80:1.8) node[nodd,fill=purple!70!blue!40]{$\kappa$} --
 +(170:1) node[leaf]{} +(0,0) --
 +(80:1) node[leaf]{} +(0,0) --
 +(-10:1) node[leaf]{} +(0,0)
++(80:-1.8) -- 
++(-15:2) node[nodd,fill=cyan!50]{$\lambda$}
 --
 +(-90:1) node[leaf]{} +(0,0) --
 %+(-15:1) node[leaf]{} +(0,0) --
 +(90:1) node[leaf]{} +(0,0)
 --
++(5:2) node[nodd,fill=brown!50]{$\mu$} --
 +(-105:1) node[leaf]{} +(0,0) --
 +(-30:1) node[leaf]{} +(0,0) --
 +(30:1) node[leaf]{} +(0,0) --
 +(105:1) node[leaf]{} +(0,0)
 (0,0) --
  +(45:1) node[leaf]{} 
 (0,0) --
 ++(120:1.7)  node[nodd,fill=blue!30]{$\gamma$} --
 +(40:1) node[leaf]{} +(0,0) --
 ++(150:1.7)  node[nodd,fill=red!50]{$\alpha$} --
 +(240:1) node[leaf]{} +(0,0) --
 +(60:1) node[leaf]{} +(0,0) --
 +(150:1) node[leaf]{} +(0,0)
 (0,0) --
 ++(190:2)  node[nodd,fill=green!40]{$\delta$} --
 +(-85:1) node[leaf]{} +(0,0) --
 ++(-180:2)  node[nodd,fill=red!50]{$\beta$} --
 %+(-270:1) node[leaf]{} +(0,0) --
 +(-140:1.1) node[leaf]{} +(0,0) --
 +(-220:1.1) node[leaf]{} +(0,0);
\end{tikzpicture}
$$
\caption{}
\label{fig1}
\end{figure}

Поскольку число узлов в дереве $T$~--- это число вершин леса $G(T)$, число мостов в~$T$~--- 
это число ребер $G(T)$, а число гроздей~--- это число компонент связности $G(T)$, имеет место равенство 
\begin{equation}\label{eq:VGB}
\mathit\Gamma=V-B.
\end{equation}

Поскольку число нелысых гроздей не больше числа нелысых узлов, имеем соотношения 
$\mathit\Gamma-L = V-B-L \leqslant V-E$. Отсюда получаем оценку на число лысых гроздей
\begin{equation}\label{eq:L}
L \geqslant E-B.
\end{equation}

% Каждому листу $o$, смежному с некоторым мостом,
% припишем пере\-ста\-нов\-ку-транс\-по\-зи\-цию $\tau_o$
% следующим образом. 
% Если соответствующая мосту квазигруппа $q$ 
% является $\{0,1\}$"=полулинейной 
% и лист $o$ является 
% его $0$-м ($1$-м, $2$-м) соседом, 
% то изотопия
% $(\tau_o, (02)(13), (02)(13))$ 
% (соответственно, $((02)(13), \tau_o, (02)(13))$,
% $((02)(13), (02)(13), \tau_o)$) 
% является автотопией квазигруппы $q$; 
% таким образом (Лемма [?????]), 
% $\tau_o$ --- либо $(0,1)$, либо $(2,3)$.
% Если же $q$ является $\{0,2\}$"=полулинейной,
% то автотопией является 
% $(\tau_o, (01)(23), (01)(23))$ 
% (соответственно, $((01)(23), \tau_o, (01)(23))$,
% $((01)(23), (01)(23), \tau_o)$);  
% в этом случае, $\tau_o$ --- 
% либо $(0,2)$, либо $(1,3)$.

Для двух различных листьев $x$ и $y$, смежных с узлами одной и той же грозди $G$ графа~$G(T)$, определим 
изотопию $\psi^{x,y}$ разложения $S$ следующим образом. Любому ребру цепи, соединяющей $x$ и~$y$ в~дереве $T$, 
сопоставим инволюцию $\xi=\xi(G)$, родную узлам грозди $G$. Каждый мост, лежащий на этой цепи, 
но не принадлежащий грозди $G$, смежен с двумя узлами грозди $G$ и~одним листом $z$. Если мост помечен 
меткой $f$, то его листовому ребру сопоставим родную транспозицию моста $\tau=\tau(z)$ 
такую, что тройка перестановок $\xi,\xi,\tau$ в подходящем порядке образует автотопию бинарной квазигруппы~$f$
(указанная транспозиция существует по лемме~\ref{l:aut3}(iv)). Оставшимся ребрам дерева~$T$ 
сопоставим тождественную перестановку.

\begin{lemma}\label{l:path}
Для любых листьев $x$ и $y$, смежных с узлами одной и той же грозди графа $G(T)$, 
изотопия $\psi^{x,y}$ является автотопией разложения $S$.
\end{lemma}
\begin{proof}
Узлы цепи $P$, соединяющей листья $x$ и $y$ в дереве $T$, делятся на два типа. Первый~--- это узлы грозди $G$. 
Если $u$ является узлом грозди $G$, то~по~построению его ребрам, принадлежащим цепи $P$, сопоставлены две 
родные узлу~$u$ инволюции $\xi(G)$, а на остальных ребрах этого узла имеем тождественные перестановки. 
По лемме~\ref{l:2rodn} такой набор перестановок образует автотопию квазигруппы, приписанной вершине $u$.

Второй тип узлов цепи $P$~--- это мосты, не принадлежащие грозди $G$, а соответствующие ее ребрам. 
Пусть $v$ является таким мостом в дереве $T$ и $z$~--- его единственный лист. По построению изотопии $\psi^{x,y}$ 
ребрам, инцидентным мосту $v$, сопоставлены две неродные для $v$ инволюции $\xi(G)$ и транспозиция $\tau(z)$, 
которые образуют автотопию квазигруппы, приписанной вершине $v$.

На каждом из узлов, не принадлежащих цепи $P$, изотопия $\psi^{x,y}$ порождает тождественную автотопию.
\end{proof}

Отметим, у лысой грозди, как и у грозди только с одним листом, отсутствуют автотопии вида $\psi^{x,y}$.

\begin{lemma}\label{l:V-B}
Если гроздь $G$ имеет $k\geqslant1$ листьев, существует (по крайней мере) $2^{k-1}$ автотопий разложения, действующих 
на листьях $G$ инволюциями, родными узлам этой грозди, или тождественно, а на листьях мостов, 
соединяющих узлы $G$,~--- родными для мостов транспозициями или тождественно. 
\end{lemma}

\begin{proof}
Пусть $x,y_1,\ldots,y_{k-1}$~--- листья грозди $G$. Автотопии $\psi^{x,y_i}$, $i=1,\ldots,k-1$, 
имеют порядок два, коммутируют и независимы друг от друга, так как для каждого $i$ только одна из них, 
а именно $\psi^{x,y_i}$, имеет нетождественную перестановку на ребре с листом $y_i$. Ими порождается 
$2^{k-1}$ автотопий, соответствующих грозди $G$.
\end{proof}

Автотопии разложения, подходящие под условие леммы~\ref{l:V-B}, будем называть \emph{порождаемыми} гроздью $G$;
все они имеют порядок $2$. При наличии вилок, мы можем также гарантировать наличие автотопий порядка $4$, которые дают дополнительный вклад в мощность группы автотопий:
\begin{lemma}\label{l:F}
Каждой вилке дерева разложения соответствует две автотопии разложения, действующие на листовых ребрах этой вилки ее родными циклами 
и тождественно на остальных ребрах.
\end{lemma}

\begin{proof}
Рассмотрим некоторую вилку $u$ разложения $S$, и пусть $\xi$~--- ее родная инволюция. Без ограничения 
общности, положим, что узел, смежный с~вилкой $u$ является ее нулевым соседом, а ее листья $x$ и $y$~--- 
первым и вторым соответственно. По лемме~\ref{l:aut3} для любой пары циклов $\varphi_1,\varphi_2$, 
родных для $u$, существует ровно одна перестановка $\xi'\in\{\mathrm{Id},\xi\}$ такая, что тройка 
$\varphi=(\xi',\varphi_1,\varphi_2)$ образует автотопию квазигруппы $f$, приписанной вилке $u$.

Если $\xi'=\mathrm{Id}$, то $\varphi$ и $\varphi^{-1}$~--- искомые автотопии $f$, которые тождественными 
перестановками дополняются до автотопий разложения $S$. Если же $\xi'=\xi$, то вместо $\varphi$ 
можно взять изотопию $(\xi',\varphi_1,\varphi_2)(\xi,\xi,\mathrm{Id})=(\mathrm{Id},\varphi_1\xi,\varphi_2)$, 
которая по лемме~\ref{l:aut3} также является автотопией квазигруппы $f$ с ее родными циклами 
$\varphi_1\xi$ и $\varphi_2$.
\end{proof}

\begin{lemma}\label{l:independence}
Нетождественные автотопии разложения, порождаемые различными гроздьями, независимы и попарно коммутируют.
\end{lemma}

\begin{proof}
Рассмотрим разложение $S$ и заметим, что две его автотопии, носители которых имеют пустое пересечение, 
независимы и коммутируют. В то же время носители автотопий, порождаемых различными гроздьями графа~$G(T)$, 
могут пересекаться только по ребрам мостов. Действительно, носитель автотопии, соответствующей вилке согласно
леммы~\ref{l:F}, ограничен листовыми ребрами этой вилки. Тогда как в носителе автотопии 
$\psi^{x,y}$, порождаемой некоторой гроздью $G$, могут содержаться лишь ребра, инцидентные узлам этой 
грозди. Если в~$G$ более одного узла, то часть указанных ребер являются одновременно внутренними 
ребрами мостов, соединяющих узлы грозди $G$.

Таким образом, достаточно доказать утверждение леммы для автотопий различных гроздей, носители которых 
пересекаются по ребрам мостов. Пусть мост $v$ соединяет узлы грозди $G$, но сам при этом содержится 
в грозди $G'$. В рамках своего носителя произвольная автотопия, порождаемая гроздью $G$, действует 
следующим образом: на ребрах, исходящих из узлов грозди $G$,~--- инволюциями, родными узлам $G$; 
на листовых ребрах мостов, соединяющих узлы грозди $G$~--- родными для мостов транспозициями.

По лемме~\ref{l:aut3}(iv) во всех автотопиях, порождаемых гроздью~$G$, участвует ровно одна из двух 
родных транспозиций моста $v$, а именно $\tau=\tau(z)$. Транспозиция $\tau$ не~может породить родную 
инволюцию моста $v$, поэтому автотопии, соответствующие грозди~$G$, и автотопии грозди $G'$ независимы.

Пусть автотопии $\theta=\psi^{x,y}$ и~$\theta'=\psi^{x',y'}$ соответствуют гроздьям $G$ и~$G'$
и их носители содержат в пересечении ребра моста $v$. На внутренних ребрах $v$ 
автотопии $\theta$ и~$\theta'$ действуют инволюциями $\xi$ и~$\xi'$, родными для $G$ и~$G'$ 
соответственно. Поскольку инволюции $(01)(23)$, $(02)(31)$ и $(03)(12)$ попарно коммутируют, 
имеем $\xi\xi'=\xi'\xi$.

На листовом ребре моста $v$ автотопия $\theta$ действует транспозицией $\tau$, 
а автотопия~$\theta'$~--- некоторой инволюцией $\eta'$ 
(либо тождественной перестановкой, что рассматривается тривиально), 
причем и $\tau$, и $\eta'$ родные для $v$. Очевидно, 
что, например, инволюция $(01)(23)$ и образующие ее транспозиции $(01)$ и (23) попарно коммутируют. 
Аналогичное справедливо и для инволюций $(02)(31)$ и $(03)(12)$. Поэтому $\tau\eta'=\eta'\tau$.

Тем самым показано, что автотопии $\theta$ и~$\theta'$ коммутируют на всех ребрах, по которым они 
могут пересекаться. Это доказывает коммутативность автотопий разложения $S$, порождаемых различными 
гроздьями.
\end{proof}

\begin{theorem}\label{th:lb}
Для произвольной $n$-арной квазигруппы $f$ порядка $4$ имеет место неравенство 
$|\mathrm{Atp}(f)|\geqslant2^{[n/2]+2}$. При $n\geqslant 5$ эта оценка точна.
\end{theorem}

\begin{proof}
Пусть квазигруппа $f$ имеет приведенное разложение $S$ с деревом разложения $T$. 
По лемме~\ref{l:V-B} для каждой грозди $G$ графа $G(T)$, имеющей $k\geqslant1$ листьев, 
можно построить $2^{k-1}$ автотопий из $\mathrm{Atp}(f)$, действующих перестановками порядка $2$ 
на переменные, соответствующие листьям грозди $G$. Суммируя по всем гроздьям графа $G(T)$, исключая лысые, 
с учетом леммы~\ref{l:independence} получаем $2^{N-(\mathit\Gamma-L)}$ автотопий квазигруппы $f$.

В дополнение, по лемме~\ref{l:F} для каждой вилки имеется $2$ автотопии из $\mathrm{Atp}(f)$, 
действующие на переменные, соответствующие листьям вилки, родными для нее циклами. Это дает множитель 
$2^F$ в число найденных автотопий квазигруппы $f$. Таким образом, с учетом~(\ref{eq:VGB}) можем записать
\begin{equation}\label{eq:NVBLF}
|\mathrm{Atp}(f)|\geqslant2^{N-(\mathit\Gamma-L)+F}=2^{N-V+B+L+F}.
\end{equation}

Пусть в дереве разложения $T$ имеется $t$ ребер и $V_s$ вершин степени $s$, $s=0,1,2,3,\ldots$. По определению 
разложения квазигруппы и принятым ранее обозначениям $V_1=N$, $V_2=0$. Тогда
$$N+\sum_{s\geqslant3}sV_s=2t=2(N+V-1).$$
Отсюда выводим 
\begin{equation}\label{eq:no-s5}
N+2V-2=\sum_{s\geqslant3}sV_s=4\sum_{s\geqslant3}V_s+\sum_{s\geqslant5}(s-4)V_i-V_3\geqslant4V-V_3.
\end{equation} 
Таким образом получаем, что $N\geqslant2V-V_3+2$.

В дереве $T$ узлы степени $3$ по числу смежных с ними листьев делятся на вилки, мосты и лысые 
(вершин, смежных только с тремя листьями нет, поскольку $n>2$). При этом число лысых 
узлов степени $3$ не превосходит общего числа лысых узлов в дереве $T$. Отсюда с учетом~(\ref{eq:L}) 
следует
\begin{equation}\label{eq:lys3} V_3\leqslant F+B+E\leqslant F+2B+L,\end{equation} 
что позволяет продолжить найденное выше неравенство для $N$:
$$N\geqslant2V-V_3+2\geqslant2V-F-2B-L+2.$$
Отсюда получаем $$-V+B\geqslant-\frac{1}{2}(N+F+L)+1.$$

Подставляя полученное в~(\ref{eq:NVBLF}), получаем
\begin{equation}\label{eq:no-lys}
 |\mathrm{Atp}(f)|\geqslant 2^{(N+F+L)/2+1}\geqslant 2^{N/2+1}=2^{(n+3)/2}
\end{equation}
(второе неравенство в (\ref{eq:no-lys}) строгое если и только если дерево разложения содержит вилку или лысую гроздь).
С учетом того, что $|\mathrm{Atp}(f)|$~--- целое число и более того, число автотопий, 
порождаемых автотопиями из лемм~\ref{l:V-B} и~\ref{l:F}, является степенью двойки, окончательно имеем
\begin{equation}\label{eq:lowerest}
	|\mathrm{Atp}(f)|\geqslant2^{\lfloor n/2 \rfloor+2}.
\end{equation}

Далее убедимся, что оценка~(\ref{eq:lowerest}) достигается на определенных квазигруппах.
Для этого рассмотрим квазигруппы $l_3^\bullet$ и $l_4^\bullet$  из леммы~\ref{l:aut4}, которые обозначим через $f$ и $h$, соответственно, и изотопную 
$3$"=квазигруппе $f$ квазигруппу $x_0=g(x_1,x_2,x_3)=\tau f (\tau x_1,\tau x_2,\tau x_3)$, где $\tau=(12)$. 

Квазигруппы $f$ и $g$ являются $\{0,2\}$- и $\{0,1\}$"=полулинейными соответственно. Их группы автотопий 
изоморфны, а именно сопряжены транспозицией~$\tau$. Родной инволюцией квазигруппы $g$ является 
перестановка $(01)(23)$, а перестановки $(01)$ и $(23)$~--- родные транспозиции $g$.

Заметим, что только тождественная перестановка $\mathrm{Id}$ входит одновременно в состав автотопий квазигрупп $f$ и $g$. 
То же справедливо и для пары квазигрупп $h$ и $g$. Поэтому, если квазигруппы $f$ и $g$ ($h$ и $g$) оказываются соседними 
в дереве разложения некоторой третьей квазигруппы $q$, то по лемме~\ref{lem:composition} их автотопии могут соединяться в~общую автотопию разложения квазигруппы $q$ 
только по тождественной перестановке.

Исходя из этих рассуждений, легко убедиться, что при нечетном $n\geqslant 5$ квазигруппа арности $n$ 
с деревом разложения вида
$$
\begin{tikzpicture}
\draw (0,0) node[nodd]{$f$} --
 +(75:1) node[leaf]{} +(0,0) --
 +(-105:1) node[leaf]{} +(0,0) --
 +(165:1) node[leaf]{} +(0,0) --
++(-15:2) node[nodd]{$g$} --
 +(-90:1) node[leaf]{} +(0,0) -- 
 +(90:1) node[leaf]{} +(0,0) -- 
++(0:1) --
++(0:1) node[nodd]{$\ldots$} --
++(0:1) --
++(0:1) node[nodd]{$f$}  --
 +(-90:1) node[leaf]{} +(0,0) -- 
 +(90:1) node[leaf]{} +(0,0) -- 
++(15:2) node[nodd]{$g$} --
 +(-75:1) node[leaf]{} +(0,0) --
 +(15:1) node[leaf]{} +(0,0) --
 +(105:1) node[leaf]{} +(0,0);
\end{tikzpicture}
$$
имеет группу автотопий порядка $2^{(n-1)/2+2}$. В то же время при четном $n\geqslant6$ квазигруппа 
арности $n$ с деревом разложения вида
$$
\begin{tikzpicture}
\draw (0,0) node[nodd]{$h$} --
 +(75:1) node[leaf]{} +(0,0) --
 +(-90:1) node[leaf]{} +(0,0) --
 +(150:1) node[leaf]{} +(0,0) --
 +(-150:1) node[leaf]{} +(0,0) --
++(0:2) node[nodd]{$g$} --
 +(-90:1) node[leaf]{} +(0,0) -- 
 +(90:1) node[leaf]{} +(0,0) -- 
++(-15:2) node[nodd]{$f$} --
 +(-90:1) node[leaf]{} +(0,0) -- 
 +(90:1) node[leaf]{} +(0,0) -- 
++(0:1) --
++(0:1) node[nodd]{$\ldots$} --
++(0:1) --
++(0:1) node[nodd]{$g$}  --
 +(-90:1) node[leaf]{} +(0,0) -- 
 +(90:1) node[leaf]{} +(0,0) -- 
++(15:2) node[nodd]{$f$} --
 +(-75:1) node[leaf]{} +(0,0) --
 +(15:1) node[leaf]{} +(0,0) --
 +(105:1) node[leaf]{} +(0,0);
\end{tikzpicture}
$$
имеет группу автотопий порядка $2^{n/2+2}$. И в том, и в другом случае в~(\ref{eq:lowerest}) 
достигается равенство.
\end{proof}
\begin{note} В случаях $n=3$ и $n=4$ указанная в теореме~\ref{th:lb} оценка не точна. 
Примеры, приведенные в доказательстве, вырождаются в полулинейные квазигруппы, у которых есть автотопия из родных транспозиций (лемма~\ref{l:aut4}(ii)),
которая не учитывалась при выводе оценки. На этих примерах также достигается минимум порядка группы автотопий, но он в два раза больше чем в формуле в теореме~\ref{th:lb}.
В случае же разделимых квазигрупп при $n=3$ и $n=4$ дерево разложений обязательно содержит вилку. Если вилки две, то несоответствие формуле видно из неравенства (\ref{eq:no-lys}),
но даже если вилка одна, то в соответствии с леммой~\ref{l:aut3}(iv) есть <<неучтенные>> автотопии с нетождественной инволюцией на внутреннем ребре. 
\end{note}

%=================================
%=================================

\subsection{Квазигруппы с минимальной группой автотопий}\label{s:min}
Примеры, приведенные в доказательстве теоремы~\ref{th:lb},  --- не единственные, на которых достигается оценка.
Таких квазигрупп довольно много, но все их можно охарактеризовать. 
В настоящем разделе мы покажем, как это сделать при нечетном $n$, 
основываясь на том, что все неравенства, использовавшиеся в рассуждении при доказательстве теоремы, 
для такой квазигруппы обращаются в равенства (при четном $n$ это не так, поскольку при получении окончательной формулы используется округление вверх).

Из доказательства нетрудно извлечь следующие свойства приведенных разложений, у которых ровно $2^{\lfloor n/2 \rfloor +2}=2^{\frac{n+3}2}$ автотопий при нечетном $n$: 
% (с границей $2^{[n/2]+2}$ при четном $n$ ситуация чуть сложнее):
\begin{enumerate}
    \item[(I)] нет вершин степени больше $4$ (из равенства в (\ref{eq:no-s5})); 
    \item[(II)] нет вилок (из равенства в (\ref{eq:no-lys}));
    \item[(III)] нет лысых гроздей (из равенства в (\ref{eq:no-lys}));
    \item[(IV)] в каждой грозди не более одной нелысой вершины (из равенства в (\ref{eq:L}));
    \item[(V)] нет лысых вершин степени больше $3$ (из равенства в (\ref{eq:lys3})).
\end{enumerate}
Последние три условия означают, что 
\begin{enumerate}
    \item[(III--V)] каждая гроздь имеет ровно одну нелысую вершину (мост или узел степени $4$), остальные вершины лысые степени $3$.
\end{enumerate}
Если при выполнении этих условий все узлы степени $4$ будут помечены $3$"=квазигруппой, изотопной $l_3^\bullet$, то у разложения не будет других автотопий, 
кроме рассмотренных в доказательстве теоремы, и оценка $2^{\frac{n+3}2}$ будет достигаться. С другой стороны, верно следующее утверждение.
\begin{lemma}\label{l:no-z4z4-z2z4}
Всем yзлам степени $4$ $n$-квазигруппы с группой автотопий порядка $2^{\frac{n+3}2}$ соответствует $3$"=квазигруппа, изотопная $l_3^\bullet$.
\end{lemma}
\begin{proof}
С точностью до изотопии, обращения и перестановки аргументов, существует ровно пять $3$"=квазигрупп 
\cite{RU:ZZ:2002:l16r13.ru}. Из них одна линейная и одна не является полулинейной. Оставшиеся три квазигруппы полулинейные, но не линейные --- это $l_3^\bullet$, $g$ и $h$, где
$$
g(x_1,x_2,x_3) = x_1\oplus (x_2 +_4 x_3)
,\qquad
h(x_1,x_2,x_3) = x_1 +_4 x_2 +_4 x_3.
$$
У $g$ есть автотопия $\phi=((01)(23),(01)(23),\mathrm{Id},\mathrm{Id})$ с двумя неродными инволюциями. У $h$ есть автотопия $((01)(23),(01)(23),(01)(23),(01)(23))$ из четырех неродных инволюций.
Предположим, что имеется $n$"=квазигруппа $f$ с группой автотопий порядка $2^{\frac{n+3}2}$ с приведенным деревом разложений, и некоторому узлу $\alpha$ приписана $3$"=квазигруппа, 
изотопная $g$ или $h$, без потери общности будем считать что это $g$ (для $h$ доказательство аналогично). Построим автотопию дерева разложения, при которой ребрам, инцидентным узлу $\alpha$, будут соответствовать перестановки набора $\phi$. Рассмотрим нулевого соседа узла $\alpha$. Если это лист, то ему припишем перестановку $(01)(23)$. Если это узел, то он принадлежит некоторой грозди. Рассмотрим произвольный лист этой грозди (согласно (III), гроздь нелысая) и путь в дереве разложения, соединяющий этот лист с узлом $\alpha$. Всем ребрам этого пути припишем перестановку $(01)(23)$. Листьям мостов, соединяющих узлы грозди, на рассматриваемом пути, припишем родную транспозицию, согласно лемме~\ref{l:aut3}(iv). Таким же образом поступим с первым соседом узла $\alpha$.
Оставшимся ребрам припишем тождественные перестановки. Полученная изотопия является автотопией дерева разложения, в чем легко убедиться, рассмотрев узлы и мосты рассмотренных гроздей. Более того, эта автотопия не входит в число автотопий, входящих в оценку теоремы~\ref{th:lb}, поскольку среди тех автотопий рядом с узлами степени $4$ не могут быть неродные инволюции. Таким образом порядок группы автотопий строго больше $2^{\frac{n+3}2}$. Случай узла с приписанной $3$"=квазигруппой $h$ рассматривается аналогично, при этом нужно рассмотреть всех четырех соседей.
\end{proof}

Деревья разложений, удовлетворяющие перечисленным условиям, можно охарактеризовать.

\begin{theorem}\label{th:min-all}
Каждая $n$-арная квазигруппа $f$ с группой автотопий порядка $2^{\frac{n+3}2}$ изотопна квазигруппе с деревом разложения, построенным следующим образом:
\begin{enumerate}
    \item[(i)] возьмем произвольное дерево с $(n-1)/2$ вершинами степени не больше $3$, вершины дерева назовем узлами, затем к каждому узлу степени $i$ присоединим $4-i$ листа;
    \item[(ii)] выделим некоторые (возможно все, возможно ни одного) узлы с одним листом, заменим каждый из них на два смежных узла степени три одним из трех возможных способов (четыре соседа узла степени $4$ можно тремя способами распределить по окрестностям новых узлов степени $3$);
    \item[(iii)] множество узлов дерева как двудольного графа разобьем на две доли, узлам одной доли припишем $\{0,1\}$"=полулинейные квазигруппы, изотопные $\mathbb{Z}_4$ и $l_3^\bullet$ (в зависимости от степени), узлам второй доли припишем $\{0,2\}$"=полулинейные квазигруппы, изотопные $\mathbb{Z}_4$ и $l_3^\bullet$;
    \item[(iv)] для завершения формальностей осталось выбрать лист, соответствующий значению квазигруппы $f$, и согласованным образом пронумеровать соседей каждого узла.
\end{enumerate}
\end{theorem}
\begin{proof}
 Рассмотрим построенное дерево разложения. Узлы степени $3$ делятся на пары (лысый узел, мост), полученные в соответствии с шагом (ii). 
 При этом каждый мост, соответствующий ребру некоторой грозди, в паре с лысым узлом этой грозди. Поскольку каждая гроздь представляет собой дерево, а у дерева число вершин на $1$ больше числа ребер,
 получаем, что у грозди ровно одна нелысая вершина. По построению, это либо мост, либо узел степени $4$. Таким образом, дерево разложений удовлетворяет условиям (I--V).
 
 Обратно, из условий (III--V) следует, что единственным образом каждому листу взаимнооднозначно соответствует соседний с ним лысый узел степени $3$ 
 (действительно, в дереве грозди одинаковое число ребер и лысых узлов). Стянув соответствующие пары вершин (операция, обратная к (ii)), 
 мы получим дерево, все узлы которого нелысые и имеют степень $4$; любое такое дерево может быть получено на шаге (i).
 Согласно леммам~\ref{l:isotop} и \ref{l:no-z4z4-z2z4}, разметка узлов дерева разложений квазигруппы с группой автотопий порядка $2^{\frac{n+3}2}$, с точностью до изотопии, соответствует разметке, полученной на шаге (iii).
\end{proof}

%С другой стороны, возможность использовать в разложении другие полулинейные $3$-квазигруппы требует дополнительного анализа
%группы их автотопий. На этом пути можно получить характеризацию всех квазигрупп с минимальной группой автотопий.

%=================================
%=================================
%=================================
\section{Верхняя оценка}\label{s:up}

В этом разделе мы докажем, что максимальное значение $6\cdot 2^n$ порядка группы автотопий $n$-арной квазигруппы порядка $4$ 
достигается на линейной квазигруппе, единственной с точностью до изотопии. 

При оценке мы будем полагаться на хорошо известную теорему об орбите и стабилизаторе, % orbit--stabilizer theorem
которая в случае группы автотопий кода означает, что мощность этой группы 
равна произведению мощности стабилизатора вершины кода (в группе автотопий) 
на мощность орбиты этой вершины (относительно группы автотопий).
Мы начнем с нескольких вспомогательных утверждений про автотопии $n$-арной квазигруппы, стабилизирующие вершину,
взяв в качестве последней набор из всех нулей.

Следующая лемма верна для произвольного порядка.
\begin{lemma}
\label{lem:symmetry}
Пусть $f$ есть $n$"=квазигруппа, и $f(0,\ldots,0)=0$. 
Автотопия $(\theta_0,\ldots,\theta_n)$ квазигруппы $f$, стабилизирующая $(0,\ldots,0)$,
определяется однозначно каждой из своих перестановок $\theta_i$, $i\in \{0,...,n\}$.
\end{lemma}
\begin{proof}
Пусть автотопия $(\theta_0,\ldots,\theta_n)$ квазигруппы $f$ стабилизирует $(0,\ldots,0)$. 
По определениям это значит, что 
$\theta_0(0)=\ldots=\theta_n(0)=0$. 
Отсюда по определению автотопии имеем
$$\theta_0^{-1} f(\theta_1 x_1,0,\ldots,0) = f(x_1,0,\ldots,0)$$
для любого $x_1$. Эквивалентно,
$$\theta_1 x_1 = f^{<1>}(\theta_0 f(x_1,0,\ldots,0)),$$
откуда мы видим, что перестановка $\theta_1$ выражается через $\theta_0$. 
Применив этот же аргумент к обращению $f^{<1>}$, мы выразим 
$\theta_0$ через $\theta_1$. Аналогичным образом $\theta_0$ связана с $\theta_2$, \ldots, $\theta_n$, 
что доказывает  утверждение леммы.
\end{proof}
\begin{corollary}
\label{c:order}
Пусть $f$ есть $n$"=квазигруппа. 
Пусть $\theta=(\theta_0,\ldots,\theta_n)$ --- автотопия квазигруппы $f$, 
стабилизирующая некоторый набор $(a_0,\ldots,a_n)$.
Тогда все перестановки $\theta_i$ имеют одинаковый порядок.
\end{corollary}
\begin{proof}
По определению $n$-арной квазигруппы, если $\theta_i(a_i)=a_i$ для каждого $i$ от $1$ до $n$,
то $\theta_0 f(a_1,\ldots,a_n)=f(a_1,\ldots,a_n)$. Таким образом, мы можем считать,
что условие следствия выполнено с $a_0=f(a_1,\ldots,a_n)$. Далее, без потери общности мы можем 
считать, что $a_0=a_1=\ldots=a_n=0$ 
(в противном случае перейдем к рассмотрению изотопной квазигруппы $f'(x_1,\ldots,x_n)=\alpha_0 f(\alpha_1 x_1,\ldots,\alpha_n x_n)$,
где $\alpha_i=(0a_i)$, и ее автотопии $(\alpha_0\theta_0\alpha_0,\ldots,\alpha_n\theta_n\alpha_n)$, сопряженной к $\theta$ и стабилизирующей 
$(0,\ldots,0)$).

По лемме~\ref{lem:symmetry}, если одна из перестановок тождественна, то тождественны и все остальные. Применив это утверждение к автотопиям 
$(\theta_0^k,\ldots,\theta_n^k)$, $k=1,2,\ldots$, получаем утверждение следствия.
\end{proof}

\begin{lemma}
\label{l:2-to-semi}
Пусть $f$ есть $n$"=квазигруппа порядка $4$, и $f(0,\ldots,0)=0$. 
Пусть $f$ имеет автотопию $(\theta_0,\ldots,\theta_n)$ порядка $2$, стабилизирующую $(0,\ldots,0)$.
Тогда $f$ полулинейна.
\end{lemma}
\begin{proof}
 По следствию~\ref{c:order} все перестановки $\theta_i$, $i=0,\ldots,n$ имеют порядок $2$. 
 Поскольку $\theta_i(0)=0$, имеем $\theta_i \in \{(12), (13), (23)\}$ для каждого $i$. 
 Без потери общности мы можем считать, что $\theta_0=\ldots = \theta_n=(23)$
 (если это не так, рассмотрим изотопную квазигруппу, для которой данный набор является автотопией).
 Рассмотрим значение $f$ в наборе $(x_1,\ldots,x_n)$ из $\{0,1\}^n$. 
  $$ f(x_1,\ldots,x_n) = \theta_0f(\theta_1x_1,\ldots,\theta_nx_n) = \theta_0f(x_1,\ldots,x_n).$$
 Мы видим, что если значение не изменяется перестановкой $(23)$, то есть оно либо $0$ либо $1$.
 Получаем, что $f$ отображает $\{0,1\}^n$ в $\{0,1\}$, то есть является полулинейной.
\end{proof}
\begin{lemma}
\label{l:3-to-lin}
Пусть $f$ есть $n$"=квазигруппа порядка $4$, и $f(0,\ldots,0)=0$. 
Пусть $f$ имеет автотопию $(\theta_0,\ldots,\theta_n)$ порядка $3$, стабилизирующую некоторый набор  $(a_0,\ldots,a_n)$.
Тогда $f$ линейна.
\end{lemma}
\begin{proof}
По следствию~\ref{c:order} каждая из перестановок, составляющих автотопию, это $3$-цикл. 

1) Если $f$ является $\{0,1\}$-, $\{0,2\}$- или $\{0,3\}$-полулинейной в каждом аргументе,
то она является в каждом аргументе $\{a,b\}$-полулинейной для любых различных $a$ и $b$ из $\Sigma$, то есть линейной.
Таким образом, утверждение леммы доказано для полулинейных квазигрупп.

2) Пусть $f$ не является полулинейной. Рассмотрим правильное разложение и соответствующее ему дерево.
Автотопия разложения, порожденная автотопией $(\theta_0,\ldots,\theta_n)$ квазигруппы, обязана иметь тот же порядок $3$.
Таким образом, эта автотопия состоит из $3$-циклов и тождественных перестановок; 
в любом случае каждая из перестановок стабилизирует некоторый элемент из $\Sigma$.
Рассмотрим произвольный нелысый узел. Автотопия приписанной этому узлу квазигруппы, порожденная автотопией разложения,
попадает под условие следствия~\ref{c:order}, следовательно, все ее перестановки суть $3$-циклы
и по п. 1) приписанная узлу квазигруппа не может быть одновременно полулинейной и нелинейной. 
Поскольку по определению узлу правильного разложения нелинейной квазигруппы не может быть приписана линейная квазигруппа, получили противоречие.
\end{proof}

\begin{theorem}
В классе $n$"=арных квазигрупп порядка $4$ только линейные квази\-группы имеет группу автотопий 
максимальной мощности, равной $6\cdot4^n$. 
Среди нелинейных $n$-арных квазигрупп порядка $4$ только полулинейные квазигруппы с группой автотопий, 
действующей транзитивно на коде квазигруппы,
 имеют группу автотопий 
максимальной мощности, равной $2\cdot4^n$. 
\end{theorem}
\begin{proof}
 Рассмотрим произвольную $n$"=квазигруппу порядка $4$, без потери общности считая, что $f(0,\ldots,0)=0$.
 По теореме об орбите и стабилизаторе мощность группы автотопий кода $M(f)$ есть 
 произведение мощности стабилизатора вершины $(0,\ldots,0)$
на мощность орбиты этой вершины (под действием группы автотопий).
 
 Мощность орбиты вершины $(0,\ldots,0)$ не превышает мощности кода, то есть $4^n$ 
 (равенство имеет место в том и только том случае, если орбита совпадает с кодом, то есть группа автотопий действует транзитивно на коде).
 
 Мощность стабилизатора вершины $(0,\ldots,0)$ равна $1$ в случае неполулинейной квазигруппы (леммы~\ref{l:2-to-semi} и \ref{l:3-to-lin}) и
 $2$ для нелинейной полулинейной (не меньше по лемме~\ref{l:aut4} и не больше  по леммам~\ref{lem:symmetry} и~\ref{l:3-to-lin}). 
 Тем самым, доказано утверждение теоремы про нелинейные квазигруппы.
 
 Осталось вычислить группу автотопий квазигруппы $l_n(x_1,\ldots,x_n) = x_1\oplus\ldots\oplus x_n$, которой изотопна любая линейная квазигруппа.
 Для произвольного набора $(a_0,\ldots,a_n)$ из $M(l_n)$ преобразование $(x_0,\ldots,x_n) \to (x_0\oplus a_0,\ldots,x_n\oplus a_n)$ переводит 
 $(0,\ldots,0)$ в $(a_0,\ldots,a_n)$ и является автотопией. Отсюда мощность орбиты набора $(0,\ldots,0)$ равна $4^n$.
 
 Мощность стабилизатора набора $(0,\ldots,0)$ не больше чем $3!$ по лемме~\ref{lem:symmetry}. 
 С другой стороны, любой набор $(\theta_*,\ldots,\theta_*)$, 
 где $\theta_*$ одна из $3!$ перестановок множества $\Sigma$ удовлетворяющих $\theta(0)=0$,
 является автотопией квазигруппы $l_n$, что легко проверяется для $n=2$ и по индукции для произвольного $n$.
 Таким образом, мощность стабилизатора  набора $(0,\ldots,0)$ равна $6$, а мощность всей группы автотопий линейной $n$-арной квазигруппы порядка $4$ 
 равна $6\cdot4^n$.
\end{proof}

В заключение отметим, что характеризация полулинейных $n$-арных квазигрупп с транзитивной группой автотопий была получена в работе~\cite{RU:KP:2016:transitive},
где доказано соответствие между такими квазигруппами и булевыми полиномами степени не более чем $2$ от $n$ переменных.

%\bibliographystyle{gost2008}
%\bibliography{k.bib}

Горкунов Евгений Владимирович, evgumin@gmail.com

Кротов Денис Станиславович, krotov@math.nsc.ru

Потапов Владимир Николаевич, vpotapov@math.nsc.ru

Институт математики им. С.Л. Соболева СО РАН,
пр. Академика Коптюга 4, Новосибирск 630090

Новосибирский государственный университет,
Пирогова 2, Новосибирск 630090